\documentclass[11pt]{article}

\usepackage[hmargin=1in,vmargin=1in]{geometry}

\usepackage{amsmath}              % For mathematical symbols
\usepackage{amssymb}              % For mathematical symbols
\usepackage{amsthm}               % For custom theorem styles
\usepackage{enumerate}            % To support \begin{enumerate}[(i)]
\usepackage{graphicx}             % For embedding figures
\usepackage{mathrsfs}             % For script fonts
\usepackage[tight]{subfigure}     % For embedding subfigures
\usepackage{url}                  % For giving the data URL

\newtheorem{theorem}{Theorem}
\newtheorem{lemma}[theorem]{Lemma}
\newtheorem{corollary}[theorem]{Corollary}

\newtheorem{conjecture}[theorem]{Conjecture}

\theoremstyle{definition}

\newtheorem*{remark}{Remark}

\newcommand{\R}{\mathbb{R}}

\newcommand{\tri}{\mathcal{T}}

\newcommand{\appropto}{\mathrel{\vcenter{
  \offinterlineskip\halign{\hfil$##$\cr
  \propto\cr\noalign{\kern2pt}\sim\cr\noalign{\kern-2pt}}}}}

\begin{document}

\title{\Large
    Computational topology and normal surfaces: \\
    Theoretical and experimental complexity bounds%
    \thanks{The first author is grateful to the
    Australian Research Council for their
    support under the Discovery Projects funding scheme
    (projects DP1094516 and DP110101104).
    This work was done while the second author was visiting
    The University of Queensland and he is thankful for the funding by
    FAPERJ for this visit.
    Computational resources used in this work
    were provided by the Queensland Cyber Infrastructure Foundation.}}
\author{Benjamin A.~Burton\thanks{%
            School of Mathematics and Physics,
            The University of Queensland,
            Brisbane, Australia.
            \texttt{bab@maths.uq.edu.au}} \\
        \and
        Jo{\~a}o Paix{\~a}o\thanks{%
            Department of Mathematics,
			Pontif{\'i}cia Universidade Cat{\'o}lica,
			Rio de Janeiro, Brazil,
            \texttt{jpaixao@mat.puc-rio.br}} \\
        \and
        Jonathan Spreer\thanks{%
            School of Mathematics and Physics,
            The University of Queensland,
            Brisbane, Australia.
            \texttt{j.spreer@uq.edu.au}}}
\date{}

\maketitle

\begin{abstract}
    In three-dimensional computational topology, the theory of normal surfaces
    is a tool of great theoretical and practical significance.
    Although this theory typically leads to exponential time algorithms,
    very little is known about how these algorithms perform in
    ``typical'' scenarios, or how far the best known theoretical bounds are
    from the real worst-case scenarios.
    Here we study the combinatorial and algebraic complexity of
    normal surfaces from both the theoretical and experimental viewpoints.
    Theoretically, we obtain new exponential lower bounds on the
    worst-case complexities in a variety of settings that are important
    for practical computation.
    Experimentally, we study the worst-case and average-case
    complexities over a comprehensive body of roughly three billion input
    triangulations.  Many of our lower bounds are the first known
    exponential lower bounds in these settings, and experimental
    evidence suggests that many of our theoretical lower bounds
    on worst-case growth rates may indeed be asymptotically tight.
\end{abstract}

%%%%%%%%%%%%%%%%%%%%%%%%%%%%%%%%%%%%%%%%%%%%%%%%%%%%%%%%%%%%%%%%%%%%%%%%
%
%   Main body of the paper
%
%%%%%%%%%%%%%%%%%%%%%%%%%%%%%%%%%%%%%%%%%%%%%%%%%%%%%%%%%%%%%%%%%%%%%%%%

%\newpage

\newcommand{\squeeze}{\vspace{-3mm}}

%%%%%%%%%%%%%%%%%%%%%%%%%%%%%%%%%%%%%%%%%%%%%%%%%%%%%%%%%%%%%%%%%%%%%%%%
%
%   Main body of the paper
%
%%%%%%%%%%%%%%%%%%%%%%%%%%%%%%%%%%%%%%%%%%%%%%%%%%%%%%%%%%%%%%%%%%%%%%%%

\section{Introduction}
\label{sec:intro}

In three-dimensional computational topology, many important problems are
solved by exponential-time algorithms: key examples include
Haken's algorithm for recognising the unknot \cite{haken61-knot}, or
breaking down a triangulated 3-manifold into its prime decomposition
\cite{jaco03-0-efficiency,jaco95-algorithms-decomposition}.  This is in
contrast to two dimensions in which many problems are solved in polynomial
time, or higher dimensions in which important topological problems
can become undecidable \cite{chiodo11-groups,markov60-insolubility}.

A common feature of many such three-dimensional algorithms, and the
source of both their solvability and their exponential running times,
is their use of \emph{normal surfaces}.  In essence, normal surfaces
are embedded 2-dimensional surfaces that intersect the surrounding
3-dimensional triangulation in a simple fashion.  Most importantly,
they describe topological features using combinatorial data, and are
thereby well-suited for algorithmic enumeration and analysis.

Amongst the most important normal surfaces are the \emph{vertex normal
surfaces}.  These correspond to the vertices of a high-dimensional
polytope (called the \emph{projective solution space}), and together they
generate the space of all possible normal surfaces within the input
triangulation.  Many topological algorithms begin by enumerating all
vertex normal surfaces in the input triangulation, and for many
problems (such as unknot recognition and prime decomposition)
this enumeration is in fact the main bottleneck for the entire algorithm.

One remarkable feature of many algorithms in three-dimensional
computational topology is that, although they have extremely large
theoretical worst-case complexity bounds, they appear to be
\emph{much} easier to solve in practice than these bounds suggest.
For example:
\begin{itemize}
    \item In 1980, Thurston asked if the Weber-Seifert dodecahedral
    space is Haken
    (the precise meaning of this is not important here)
    \cite{birman80-problems}.
    This long-standing question became a symbolic benchmark for
    computational topology, and was only resolved by
    computer proof after 30~years \cite{burton12-ws}.
    At the heart of the proof was an enumeration of all vertex normal
    surfaces in an $n=23$-tetrahedron triangulation:
    despite a prohibitive $O(16^n \times \mathrm{poly}(n))$-time
    enumeration algorithm (the best available at the time) and
    a best known bound of
    $O(3.303^n)$ vertex normal surfaces \cite{burton11-asymptotic},
    the enumeration ran in just
    $5\frac12$~hours with only 1751 vertex normal surfaces in total.

    \item The problem of unknot recognition is of particular interest.
    Modern derivatives of Haken's original algorithm
    \cite{jaco95-algorithms-decomposition} have an
    exponential time complexity \cite{hass99-knotnp},
    but there is a growing discussion as to whether a faster
    algorithm might exist \cite{dunfield11-spanning,hass12-conp}.
    Certainly unknot recognition lies in \textbf{NP} \cite{hass99-knotnp},
    and also \textbf{co-NP} if we assume the generalised Riemann hypothesis
    \cite{kuperberg11-conp}; moreover, recent algorithmic developments
    based on linear programming now exhibit an experimental polynomial-time
    behaviour \cite{burton12-regina}.
    Deciding whether unknot recognition has a worst-case polynomial-time
    solution is now a major open problem in computational topology.
\end{itemize}

This severe gap between theory and practice is still poorly understood.
There appear to be two causes:
(i)~the best theoretical complexity bounds are far from tight;
(ii)~``pathological'' inputs that exhibit high-complexity behaviour are
rare, with ``typical'' inputs often far easier to work with.

Proving such claims mathematically remains extremely elusive.
Obtaining tight complexity bounds requires a deep interaction between
topology, normal surfaces and polytope theory, and it is difficult to
avoid making very loose estimates in at least one of these areas.
Understanding ``typical'' behaviour (such as average- or
generic-case complexity) is hampered by our very limited
understanding of random 3-manifold triangulations: even the simple task
of generating a random 3-manifold triangulation with $n$ tetrahedra has
no known sub-exponential-time solution \cite{reiner12-mfo-problems}.
In this setting, experimental work plays a crucial role in understanding
the realistic performance of algorithms, as well as the innate
difficulty of the problems that they aim to solve.

In this paper we
focus our attention on the problem of enumerating all vertex normal surfaces
within a given $n$-tetrahedron input triangulation:
as mentioned earlier, this is a central component---and often the main
bottleneck---of many algorithms in computational 3-manifold topology.
Enumeration algorithms are still evolving
\cite{burton12-regina,burton10-tree}, and they are often
hand-tailored to a particular topological problem of interest.
For this reason we do not focus on the complexity of any specific algorithm,
but instead we study two aspects of normal surface theory
that affect and constrain all of these algorithms:
\begin{itemize}
    \item \emph{Combinatorial complexity:}
    We study the \emph{total number} of vertex normal surfaces within
    the input triangulation $\tri$, which we denote by
    $\sigma(\tri)$.  This is our main quantity of interest.
    It yields an immediate \emph{lower bound} for the
    time complexity of any enumeration algorithm, since it determines
    the output size.\footnote{%
        Specifically, since each vertex normal surface can be described
        in $O(n)$ space \cite{hass99-knotnp}, the output size is
        $O(\sigma(\tri) \times n)$.}
    Moreover, $\sigma(\tri)$ also factors into \emph{upper bounds}, since
    modern enumeration algorithms are designed to run faster in situations
    where $\sigma(\tri)$ is small \cite{burton10-tree}.\footnote{%
        The tree traversal enumeration algorithm (the current state of the art)
        has running time $O(4^n \sigma(\tri) \times \mathrm{poly}(n))$.}

    \item \emph{Algebraic complexity:}
    As detailed in Section~\ref{sec:prelim}, each normal surface is
    described by a non-negative integer vector in $\R^{7n}$ (or in some
    settings, $\R^{3n}$).  We investigate the \emph{maximum coordinate}
    of this vector of any vertex normal surface within the input triangulation $\tri$,
    which we denote by $\kappa(\tri)$.
    This quantity is important for the \emph{implementation} of
    enumeration algorithms, since
    it affects whether we can work with fast native machine integer
    types or whether we must fall back to significantly more expensive
    arbitrary-precision integer arithmetic \cite{burton10-tree}.
    Moreover, $\kappa(\tri)$ features in algorithms that
    extend or even avoid the enumeration problem:
    \begin{itemize}
        \item Some algorithms, such as recognising small Seifert fibred
        spaces \cite{rubinstein04-smallsfs},
        require the complete enumeration of not just vertex normal
        surfaces but a much larger ``lattice'' of normal surfaces whose
        size is a function of $\kappa(\tri)$.
        \item Some algorithms, such as determining the crosscap number
        of a knot \cite{burton11-crosscap}, avoid vertex enumeration
        entirely by solving an integer program instead; here the bounds
        on $\kappa(\tri)$
        feature as coefficients in the integer program, and directly
        affect whether the program can be solved using off-the-shelf
        integer programming software.
    \end{itemize}
\end{itemize}

In summary, by focusing our attention on the quantities
$\sigma(\tri)$ and $\kappa(\tri)$, we learn not only about the behaviour
of current enumeration algorithms, but also about the intrinsic limits
and behaviour of the problem that they seek to solve.

We approach these combinatorial and algebraic quantities $\sigma(\tri)$
and $\kappa(\tri)$ through both theory and experiment.
Theoretically, we construct infinite ``pathological'' families of
triangulations in Section~\ref{sec:lower}
that establish \emph{exponential lower bounds} on the
worst-case scenario for both $\sigma(\tri)$ and $\kappa(\tri)$.
Experimentally, we examine both the \emph{worst case} and
\emph{average case} behaviour of these quantities
in Section~\ref{sec:expt}, using a comprehensive census of billions of
input triangulations.

Such results are highly important for practitioners in
three-dimensional computational topology, particularly given the
exponential nature of many key algorithms.  Despite this, just one
preliminary study of this type appears in the literature
\cite{burton10-complexity}.
This scarcity of results has two causes:
\begin{enumerate}[(i)]
    \item
    \emph{The lack of large, comprehensive censuses of both ``typical'' and
    ``atypical'' triangulations.}

    There are many censuses of 3-manifold triangulations in the
    literature, but these typically focus on well-structured
    triangulations with special properties (such as minimal
    triangulations, or irreducible manifolds).  Such triangulations
    are often easy to work with \cite{jaco03-0-efficiency},
    and offer little insight
    into an algorithm's worst-case (or even average-case) behaviour.

    It is only recently that large, comprehensive bodies of census data
    have been developed to study \emph{all} triangulations of a given
    input size \cite{burton10-complexity,burton11-genus}.
    By using such censuses for our experimental data, we ensure that we
    identify pathological cases, and also gain a clear understanding of
    how common or rare they are.

    \item
    \emph{The intense computation required to
    study normal surfaces with such large bodies of data.}

    Normal surface enumeration algorithms have enjoyed significant
    advances in recent years, and modern algorithms now run many orders
    of magnitude faster than their earlier counterparts
    \cite{burton12-regina,burton10-tree}.
    The experimental work in this paper required several years of
    combined CPU time, and without recent algorithmic advances
    \cite{burton09-convert,burton11-genus,burton12-regina} this work would not have been possible.
\end{enumerate}

The preliminary study in \cite{burton10-complexity} examines only the
combinatorial complexity $\sigma(\tri)$, and works with a data set of
roughly 150~million triangulations of closed 3-manifolds.
The study in this paper is significantly richer, both in scope and detail:
\begin{itemize}
    \item We examine the algebraic complexity $\kappa(\tri)$ in addition to
    the combinatorial complexity $\sigma(\tri)$;

    \item We work with a comprehensive data set of almost
    \emph{three~billion} triangulations, spanning both \emph{closed}
    manifolds (which are important for algorithms such as prime decomposition)
    and \emph{bounded} manifolds (which are important for knot algorithms);

    \item We also examine these quantities in ``optimised'' settings
    that arise in practical computation---in particular,
    \emph{one vertex triangulations} (a common optimisation used in many
    topological algorithms), and the restricted problem of
    enumerating only \emph{vertex normal discs} (which is important for
    unknot recognition, or testing surfaces for incompressibility).
\end{itemize}

Our pathological families yield the first known explicit exponential
lower bounds on worst-case complexity for the computationally important
settings of
bounded triangulations, closed 1-vertex triangulations, and normal discs.
In many settings our pathological families match the
experimental worst-case growth rates precisely, and we conjecture that the
resulting exponential bounds are in fact exact.

Of related note is a result of Hass et~al.\ \cite{hass03-spanningdisks},
who establish an exponential lower
bound on the worst-case complexity of a triangulated disc spanned by the
unknot in $\R^3$ (this has particular relevance for the complexity of unknot
recognition).  Their result operates under stricter geometric
constraints, and it is not yet known how it translates to the
more flexible setting of normal surfaces.

We emphasise again that our experimental data sets use exhaustive
censuses of \emph{all} possible input triangulations below a given size.
This paper introduces the first such census of bounded
3-manifold triangulations in the literature,
totalling over 20~billion triangulations of size $n \leq 9$.

We use exhaustive censuses because there is no known efficient algorithm
for randomly sampling large triangulations
\cite{dunfield06-random-covers,reiner12-mfo-problems}, and although
there are other methods for generating random 3-manifolds
\cite{dunfield06-random-covers,maher10-random-heegaard},
nothing is known about the bias of the resulting sample of triangulations.
As a result, although our \emph{census} is large, the
triangulations it contains are all relatively small.
Nevertheless, there are strong reasons to believe that our experimental results
are indicative of behaviour for larger inputs;
we discuss this further in Section~\ref{sec:conc}.

\section{Preliminaries}
\label{sec:prelim}

By a {\em triangulation} $\tri$, we mean a collection of $n$ abstract
tetrahedra $\Delta_i = i(0123) $, $1 \leq i \leq n$, some or all of whose faces are 
affinely identified or ``glued together'' in pairs;
here $(0123)$ refers to the four vertices of tetrahedron $\Delta_i$.
As a consequence of these face 
gluings, many tetrahedron edges may become identified together; we refer to the result as
a single {\em edge of the triangulation}, and likewise with vertices.
The gluings must be in a way that no edge is identified with itself in 
reverse as a result. Moreover, each tetrahedron face must be
identified with one and only one partner (we call these
\emph{internal faces}), or with nothing at all (we call these
\emph{boundary faces}). The set of boundary faces is called the {\em boundary of the triangulation} and
 denoted by $\partial \tri$. If $\partial \tri = \emptyset$ then $\tri$ is called a 
{\em closed triangulation}, otherwise it is said to be {\em bounded}.
Not all triangulations (closed or bounded) represent $3$-manifolds.
However, unless otherwise specified, this is true for all triangulations 
presented in this paper.

Throughout this article, the gluings of the triangles are given by a 
bijection of their vertices $i(abc) \mapsto j(de\!f)$ where the 
symbol $i(abc)$ denotes triangle $(abc)$ from tetrahedron $i$ 
and the order of the vertices determines the gluing.
A triangulation as defined above is sometimes referred to as {\em generalised
triangulation};
these are more general and flexible than simplicial complexes.
An important case is a \emph{1-vertex triangulation}, in which all
tetrahedron vertices become identified together.

The \emph{face pairing graph} of a triangulation $\tri$
is the multigraph whose nodes represent tetrahedra, and whose arcs
represent pairs of tetrahedron faces that are glued together.
A face pairing graph may contain loops (if two faces of the same
tetrahedron are glued together), and/or multiple edges (if two tetrahedra
are joined together along more than one face). See Figure 
\ref{fig:face_pairing_graphs} for examples.

\begin{figure}[tb]
    \centering
    \includegraphics[width=.65\textwidth]{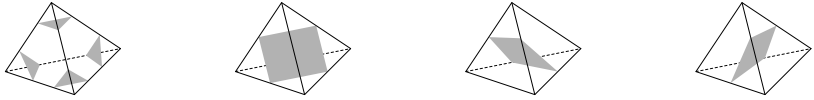}
    \caption{Normal triangles and quadrilaterals within a tetrahedron.}
	\label{fig:normalSubsets}
\end{figure}

A \emph{properly embedded} surface in $\tri$ is a surface
$\mathbf{s} \subseteq \tri$ with no self-intersections, and whose
boundary lies entirely within $\partial \tri$.
A {\em normal surface} in $\tri$ is a properly embedded
surface that meets each tetrahedron $\Delta$ of $\tri$ in a disjoint
collection of triangles and quadrilaterals, each running between
distinct edges of $\Delta$, as illustrated in Figure
\ref{fig:normalSubsets}. There are four \emph{triangle types} and three
\emph{quadrilateral types} in $\Delta$
according to which edges they meet. Within each
tetrahedron there may be several triangles or quadrilaterals of any
given type; collectively these are referred to as {\em normal pieces}.
The intersection of a normal piece of a tetrahedron with one of its
faces is called {\em normal arc}; each face has three \emph{arc types}
according to which two edges of the face an arc meets.

Counting the number of pieces of each type for
a normal surface $\mathbf{s}$ gives rise to a $7$-tuple per 
tetrahedron of $\tri$ and hence a $7n$-tuple of non-negative integers describing
$\mathbf{s}$ as a point in $\mathbb{R}_{\geq 0}^{7n}$, called its {\em normal coordinates}. 
Such a point must satisfy a set of linear homogeneous
{\em matching equations} (one for each arc type of each internal face).
These equations are necessary but not sufficient: the normal
coordinates must also satisfy a set of combinatorial constraints called
the \emph{quadrilateral constraints}, which we discuss further in the
appendix.

The solution set to the matching equations in $\mathbb{R}_{\geq 0}^{7n}$
is a polyhedral cone (the cross-section polytope
of this cone is also known as the {\em projective solution space}).
A {\em vertex normal surface} is one whose normal
coordinates lie on an extremal ray of this polyhedral cone and,
in addition, its normal coordinates are minimal for all integer points
on this ray. Thus, there are only finitely many such vertex normal surfaces;
every normal surface can then be expressed as a positive rational linear
combination of these surfaces just like every point in a polyhedral cone is
a positive rational linear combination of points on its extremal rays.
This is why,
when enumerating normal surfaces in a triangulation, we typically just consider
the finite set of vertex normal surfaces.

\section{Theoretical lower bounds}
\label{sec:lower}

Here we establish lower bounds for the worst-case
values of $\sigma(\tri)$ and $\kappa(\tri)$, i.e.,
the maximum possible $\sigma(\tri)$ or $\kappa(\tri)$ for an
$n$-tetrahedron triangulation $\tri$.
Recall that $\sigma(\tri)$ measures the combinatorial complexity, i.e., the number of vertex normal surfaces within $\tri$,
and that $\kappa(\tri)$ measures the algebraic complexity, i.e., the maximum coordinate of any vertex normal surface of $\tri$.

Few such results are known:
there are no explicit lower bounds
on the worst-case $\kappa(\tri)$ in the
literature, and the only explicit lower bound on the worst-case
$\sigma(\tri)$ is given by a family of closed triangulations
with $\sigma(\tri) \in \Theta(17^{n/4}) \simeq \Theta(2.03^n)$
\cite{burton10-complexity}.
In this section, we give new exponential
lower bounds for $\sigma (\mathcal{T})$ and $\kappa (\mathcal{T})$
in a variety of settings that hold particular relevance
for key algorithms in computational geometry and topology.
We sketch the main constructions and results here; see the
appendix for detailed proofs.

\subsection{Closed triangulations with many normal surfaces}
\label{sec:closedTrigs}

Important $3$-manifold algorithms such 
as prime decomposition often begin by converting the input triangulation
to a 1-vertex triangulation, whereupon the subsequent processing becomes
significantly easier \cite{jaco03-0-efficiency}.
The $\Theta(2.03^n)$ family of \cite{burton10-complexity} is not of this
type (each triangulation has $n+1$ vertices), which raises the question
of how such bounds behave in a 1-vertex setting:

\begin{theorem}
    \label{thm:pow2}
    There is a family $\mathcal A_n$, $n \ge 1$, of closed 1-vertex
    triangulations with $n$ tetrahedra and
    $\sigma(\mathcal A_n) = 2^n$ vertex normal surfaces.
\end{theorem}

We call these triangulations \emph{binomial triangulations}, because
more precisely they have $\binom{n}{k}$ vertex normal surfaces
of genus $k$ for each $k=0,\ldots,n$ (whereby $\sigma(\mathcal A_n) =
\sum \binom{n}{k} = 2^n$).  We construct each $\mathcal A_n$ from $n$
tetrahedra $\Delta_1,\ldots,\Delta_n$ in the following manner.

We begin by folding together two faces of the same tetrahedron $\Delta_i
= i(0123)$ by the gluing $i(012) \mapsto i(013)$ for $i=1,...,n$. Then we identify
tetrahedra $\Delta_i$ and $\Delta_{i+1}$ by $i(123) \mapsto (i+1)(230)$
(where $\Delta_{n+1} = \Delta_{1}$). These gluings identify all the
vertices to a single vertex, therefore $\mathcal A_n$ is a 1-vertex triangulation. See Figure \ref{fig:face_pairing_graphs} for a picture of the face pairing graph of $\mathcal A_n$ for $n=6$.
It can be shown that each $\mathcal A_n$ is a closed 1-vertex
triangulation of the 3-sphere.
% TODO (see the appendix for the details). -- Prove this in the journal version.

%\includegraphics[width=12.5cm]{binomial_construction}\label{binomial_construction}
%\begin{figure}[htb]
%\centering
%\includegraphics[width=10cm]{binomial_construction.png}
%\caption{Construction of binomial family $\mathcal A_n$ with $n$ tetrahedra}
%\label{fig:binomial_construction}
%\end{figure}

To see why $\sigma(\mathcal A_n) = 2^n$, we observe that
in each tetrahedron there are two normal ``subsurfaces'' which are compatible with any other normal 
surface of the triangulation. One of these adds genus to the overall surface 
and the other does not. We show that the vertex normal
surfaces are precisely combinations of these subsurfaces, whereby
the binomial coefficients and $2^n$ growth rate easily follow.

\begin{remark}
Experimentation suggests that this family $\mathcal A_n$ might in fact
yield a tight upper bound
for closed 1-vertex triangulations; see Section~\ref{sec:expt} for details.
\end{remark}

\subsection{Bounded triangulations with many normal surfaces}
\label{sec:bdTrigs}

%%%%%%%%%%%%%%%%
%%% JONATHAN %%%
%%%%%%%%%%%%%%%%

The number of vertex normal surfaces in a bounded triangulation
has a direct impact on algorithms such as unknot recognition and
incompressibility testing \cite{jaco95-algorithms-decomposition}.
Here we give the first explicit exponential lower bound on the
worst-case growth rate of this quantity.
The proof is based on a general construction principle
(Lemma~\ref{lem:treeLemma}) which, for an arbitrary bounded triangulation
$\mathcal{G}_0$ satisfying certain weak constraints, uses a
recursive squaring argument to obtain a family of triangulations
$\{\mathcal{G}_k\}$ with $\Omega(\beta^n)$ vertex normal surfaces,
where the exponential base $\beta$ is derived from $\mathcal{G}_0$.
By choosing a suitable
starting triangulation $\mathcal{G}_0$, we obtain the explicit
base $\beta \simeq 2.3715$ (Corollary~\ref{cor:bdrygrowth}).

For both unknot recognition and incompressibility testing, we can
improve the underlying algorithms by only considering vertex normal
\emph{discs} (vertex normal surfaces that are topologically trivial).
In Theorem~\ref{thm:pathFamily} we show that this restricted
quantity is also worst-case exponential: we build a family of
triangulations with $\Theta(2^n)$ vertex normal discs.

\begin{lemma}
	\label{lem:treeLemma}
	Suppose $\mathcal{G}_0$ is a bounded triangulation with $n_0$ tetrahedra,
    $f_0$ is a boundary face of $\mathcal{G}_0$ such that not all
	vertices of $f_0$ are identified in $\mathcal{G}_0$, and
    $c_0$ is one of the three normal arc types on $f_0$.
    If there are $\alpha_0$ vertex normal surfaces in $\mathcal{G}_0$
    that meet $f_0$ in at least one arc of type $c_0$ but in no other normal arc
    types, then $\mathcal{G}_0$ can be extended to a family of
    triangulations $\{\mathcal{G}_k\}$ in which the number of
    vertex normal surfaces grows at a rate of $\Omega(\beta^n)$,
    where $\beta = \alpha_0^{1/(n_0+1)}$.
\end{lemma}

In the proof (for a detailed proof see the appendix), we recursively construct
$\mathcal{G}_k$ by joining two copies of $\mathcal{G}_{k-1}$
to an additional tetrahedron $\Delta_k$ along their faces $f_{k-1}$
(see Figure \ref{fig:attachingMapSimple}).
Suppose there are $\alpha_{k-1}$ vertex normal surfaces in
$\mathcal{G}_{k-1}$ that meet $f_{k-1}$ in only arcs of type
$c_{k-1}$.  For each pair of such surfaces in the two copies of
$\mathcal{G}_{k-1}$,
we can combine these surfaces in a way that extends through $\Delta_k$
to meet one of the free faces $f_k$ of $\Delta_k$ in just one chosen normal arc
type $c_k$, and this extension yields a vertex normal surface of
$\mathcal{G}_k$.
There are $\alpha_{k-1}^2$ such pairings, and
therefore $\alpha_k \geq \alpha_{k-1}^2$ such vertex normal surfaces of
$\mathcal{G}_k$.
This recurrence yields the final growth rate of
$\Omega(\beta^n)$ where $\beta = \alpha_0^{1/(n_0+1)}$.
We need the assumption that not all vertices of the boundary face $f_0$
are identified in $\mathcal{G}_0$ to show that each triangulation
$\mathcal{G}_k$ represents a bounded 3-manifold.
	\begin{figure}[tb]
		\begin{center}
		   \includegraphics[width=.35\textwidth]{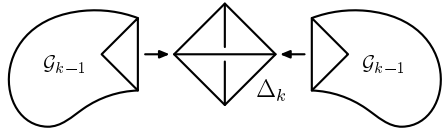}
		   \caption{Attaching two copies of $\mathcal{G}_{k-1}$ to the tetrahedron $\Delta$.\label{fig:attachingMapSimple}}
		\end{center}
	\end{figure}

\begin{corollary}
    \label{cor:bdrygrowth}
	There is a triangulation $\mathcal{G}=\mathcal{G}_0$
    that is the starting point of a family of bounded triangulations 
	$\{\mathcal{G}_k\}$, $k \geq 0$, with $n_k = 12 \cdot 2^k - 1$ tetrahedra and $\sigma (\mathcal{G}_k) \geq 2.3715^{n_k}$.
\end{corollary}

We prove this using Lemma~\ref{lem:treeLemma} by choosing a starting
triangulation $\mathcal{G} = \mathcal{G}_0$ with
$n_0=11$ tetrahedra and a choice of face $f_0$ and arc type $c_0$ with
$\alpha_0 = 31\,643$ corresponding vertex normal surfaces.
This yields a growth rate of $\Omega(\beta^n)$ with
$\beta = 31\,643^{1/12} \simeq 2.3715$.
The face pairing graph of $\mathcal{G}$ is shown in Figure
\ref{fig:face_pairing_graphs}; for a detailed construction
see the appendix.

\begin{figure}[tb]
	\centering
	\includegraphics[width=.9\textwidth]{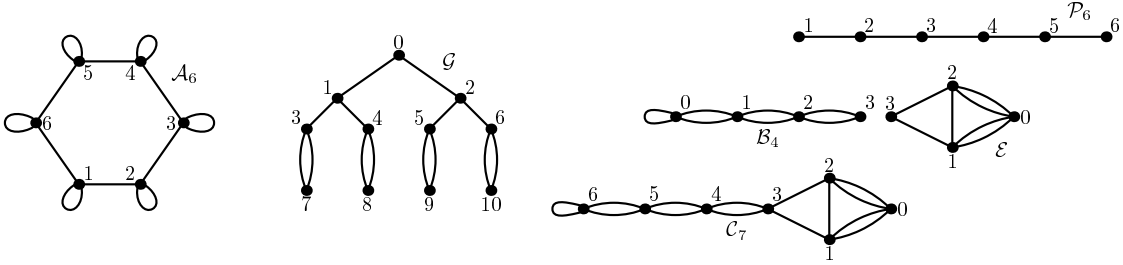}
	\caption{Face pairing graphs of the binomial family $\mathcal{A}_6$,
    triangulation $\mathcal{G}$,
		the path family $\mathcal{P}_6$, the bounded family $\mathcal{B}_4$,
		triangulation $\mathcal{E}$, and the closed family $\mathcal{C}_7$.}
	\label{fig:face_pairing_graphs}
\end{figure}

%%%%%%%%%%%%
%%% JOAO %%%
%%%%%%%%%%%%

%\begin{figure}[htb]
%\centering
%\includegraphics[width=10cm]{root.png}
%\caption{Gluing of $T_{k-1}^1$ and $T_{k-1}^2$ to the root $\Delta_k$}
%\label{fig:root}
%\end{figure}

\smallskip
For our final result, we construct the
\emph{path triangulation} $\mathcal{P}_n$ from $n$ tetrahedra
$\Delta_i = i(0123) $, $i=1,...,n$, by joining tetrahedra
$\Delta_i$ and $\Delta_{i+1}$ by the map $i(012) \mapsto (i+1)(013)$.
It can be shown that each $\mathcal P_{n}$ is a bounded triangulation
whose underlying 3-manifold is the 3-ball (for details see the appendix).

\begin{theorem}
	\label{thm:pathFamily}
For each $n \ge 1$, $\mathcal P_{n}$ has $2^{n+1}+\frac{(n+1)(n+2)}{2}
\in \Theta(2^n)$ vertex normal discs.
\end{theorem}

We prove this by obtaining explicit recurrences for the number of
vertex normal surfaces with different normal arcs based on the
matching equations. We
have essentially two choices for each normal arc, and either a triangle or
quadrilateral can be added to the previous triangulation giving the
$\Theta(2^{n})$ growth rate.  Again, see the appendix for details.

\subsection{Lower bounds for the size of normal coordinates}
\label{sec:largeCoords}

Here we give exponential lower bounds on the
worst-case algebraic complexity $\kappa ( \mathcal{T} )$.
Our bounds follow a Fibonacci growth rate of
$\Omega([(1+\sqrt{5})/2]^n) \simeq \Omega(1.618)^n$. 
Understanding the worst-case $\kappa(\tri)$ is important for improving the
time and space complexity of normal surface enumeration algorithms
due to a better handling of the integer arithmetic involved
(see Section~\ref{sec:intro}). 

To obtain such lower bounds, we first construct a family of bounded
triangulations, each containing a vertex normal disc with coordinates
growing exponentially in the number of tetrahedra.
We then close the
bounded family using a constant number of additional tetrahedra so
that the vertex normal surface with exponential coordinates is
preserved.
In this way, we are able to construct two families of triangulations, bounded and closed, with Fibonacci type growth rates 
for $\kappa (\tri)$.

The key objects of the construction are so-called
\emph{layered solid tori} \cite{jaco03-0-efficiency}.
These are parameterised triangulations of the solid torus:
the layered solid torus denoted
$\operatorname{LST} (a,b,a+b)$ has as its boundary
a triangulation of the torus with exactly three boundary edges,
such that the {\em meridian disc}
(the unique disc of the solid torus meeting the boundary
in a non-contractible closed curve) intersects the boundary edges
in $a$, $b$ and $a+b$ points.

Layered solid tori are very common tools when constructing triangulations of a given type of $3$-manifold 
(see \cite{kirby78-calculus,lickorish97,orlik72-seifert} for more about constructing $3$-manifolds). 
The most prominent example of a layered solid torus, the one 
tetrahedron triangulation of $\operatorname{LST} (1,2,3)$, is shown in Figure \ref{fig:lst123}.

\begin{figure}[tb]
    \begin{center}
            \includegraphics[width=.8\textwidth]{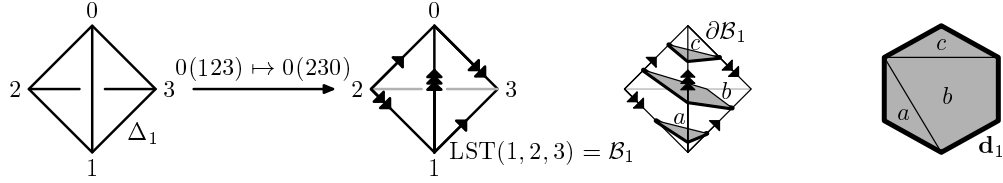}
	    \caption{Left: construction of the $1$-tetrahedron triangulation of $\operatorname{LST} (1,2,3)$. Right: the meridian 
		disc $\mathbf{d}_1$ which is a $6$-gon with $1$ normal quad, $2$ normal 
		triangles and a maximum of $3$ edge intersections.\label{fig:lst123}}
     \end{center}
\end{figure}

\begin{theorem}
	\label{thm:boundedLargeCoords}
	There is a family $\mathcal{B}_{n}$ of bounded $1$-vertex
    triangulations with $n$ tetrahedra,
    where $\mathcal{B}_n$ contains a vertex normal disc $\mathbf{d}_{n}$ with
    maximum coordinate $\operatorname{F}_{n+1}$, 
	where $\operatorname{F}_k$ denotes the $k$-th Fibonacci number.
\end{theorem}

The family $\mathcal{B}_{n}$ consists of layered solid tori of type
$\operatorname{LST}
(\operatorname{F}_{n+1},\operatorname{F}_{n+2},\operatorname{F}_{n+3})$
and each vertex 
normal surface $\mathbf{d}_{n}$ is the corresponding meridian disc.
See the appendix for details of the proof.

The key idea for the construction of a closed family $\mathcal{C}_{n}$
of $n$-tetrahedron triangulations containing a 
vertex normal surface with exponentially growing coordinates is to find a small $m$-tetrahedron triangulation 
$\mathcal{E}$ with the same boundary as the family of layered solid tori $\mathcal{B}_{n}$ acting like a type of 
plug. By this we mean that $\mathcal{E}$ contains a normal surface intersecting $\partial \mathcal{E}$ in the same way as
$\mathbf{d}_{n}$ intersects $\partial \mathcal{B}_{n}$ for all $n \geq 1$. This  
gives rise to a vertex normal surface in $\mathcal{C}_{n}$ with maximum
coordinate greater than or equal to $\operatorname{F}_{n+1}$. 
Since $m$ is constant, this gives the same asymptotic lower bound
on $\kappa(\tri)$ for closed triangulations as for bounded triangulations.

\begin{remark}[The triangulation $\mathcal{E}$]
	There is a $4$-tetrahedron bounded triangulation $\mathcal{E}$ with 
	$\partial \mathcal{E} = \partial \mathcal{B}_m$ for all $m$ which contains two vertex normal surfaces 
	$\mathbf{s}$ and $\mathbf{t}$ that can be combined into a normal surface intersecting $\partial \mathcal{E}$
	in the same pattern as $\mathbf{d}_m$. The face pairing graph of $\mathcal{E}$ is shown in Figure 
	\ref{fig:face_pairing_graphs}, and the intersection of $\mathbf{s}$ and $\mathbf{t}$ with its 
	boundary $\partial \mathcal{E}$	is shown in Figure \ref{fig:sAndt}. A detailed description of $\mathcal{E}$
	can be found in the appendix.
	\begin{figure}[tb]
		\begin{center}
	       	   \includegraphics[width=.3\textwidth]{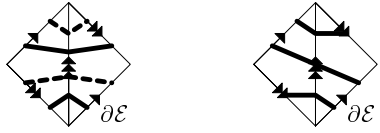}
	 	   \caption{Intersection of $\mathbf{s}$ (left) and $\mathbf{t}$ (right) with $\partial \mathcal{E}$ \label{fig:sAndt}}
		\end{center}
	\end{figure}
\end{remark}

\begin{theorem}
	\label{thm:closedLargeCoords}
	There is a family $\mathcal{C}_{n}$ of closed $1$-vertex triangulations with $n$ tetrahedra, $n \geq 5$, each 
	containing a vertex normal surface with maximum coordinate at least $\operatorname{F}_{n-3}$
	if $n \equiv 2 \bmod 3$ or at least $2\operatorname{F}_{n-3}$ otherwise.
\end{theorem}

As outlined above, we construct $\mathcal{C}_{n}$ by gluing
$\mathcal{B}_{n-4}$ and $\mathcal{E}$ along their boundary tori.
If $n \equiv 2 \bmod 3$, the meridian disc 
$\mathbf{d}_{n-4}$ glued to a combination of $\mathbf{s}$ and $\mathbf{t}$ yields a vertex normal projective 
plane or, if $n \equiv 0,1 \bmod 3$, twice $\mathbf{d}_{n-4}$
with a combination of $\mathbf{s}$ and $\mathbf{t}$ yields a vertex normal 
sphere; the maximum coordinates are then as stated.
See the appendix for details.

We note that the vertex normal sphere from above (in the case
$n \equiv 0,1 \bmod 3$) is the only non-vertex linking
normal sphere in $\mathcal{C}_{n}$. Detecting these normal surface types is one of the key tasks in important $3$-manifold problems
such as prime decomposition. Hence, the family $\mathcal{C}_{n}$ is an example for a case
where, in order to prove the existence of such a normal sphere, dealing with exponentially 
large normal coordinates cannot be avoided. This is a hint towards the conjecture that these problems are intrinsically
hard to solve using normal surface enumeration methods.

\section{Experimental behaviour}
\label{sec:expt}

\begin{table}[tb]
\centering
\small
\begin{tabular}{c|r|r|r}
Input size $n$ &
\multicolumn{1}{c|}{Closed} &
\multicolumn{1}{c|}{Closed and 1-vertex} &
\multicolumn{1}{c}{Bounded} \\
\hline
1   &                4 &                3 &                 3 \\
2   &               17 &               12 &                17 \\
3   &               81 &               63 &               156 \\
4   &              577 &              433 &            2\,308 \\
5   &           5\,184 &           3\,961 &           45\,046 \\
6   &          57\,753 &          43\,584 &          995\,920 \\
7   &         722\,765 &         538\,409 &      25\,225\,447 \\
8   &      9\,787\,509 &      7\,148\,483 &     695\,134\,018 \\
%\cline{4-4}
9   &    139\,103\,032 &     99\,450\,500 & 19\,933\,661\,871 \\
10  & 2\,046\,869\,999 & 1\,430\,396\,979 & \\
\hline
Total&$n \leq 10$: 2\,196\,546\,921 & $n \leq 10$: 1\,537\,582\,427 &
    $n \leq 8$: \phantom{00}\,721\,402\,915 \\
& & & $n \leq 9$: 20\,655\,064\,786
\end{tabular}
\caption{The number of 3-manifold triangulations in the census}
\label{tab-census}
\end{table}

We turn now to an experimental study of the combinatorial and
algebraic complexities of vertex normal surfaces.
Our experimental data consists of \emph{all closed 3-manifold
triangulations} with $n \leq 10$~tetrahedra, and
\emph{all bounded 3-manifold triangulations} with $n \leq 8$~tetrahedra
(each appearing precisely once up to relabelling).
As shown in Table~\ref{tab-census}, this yields almost 3~billion
triangulations in total
($2\,196\,546\,921 + 721\,402\,915 = 2\,917\,949\,836$).
We also extract the $\sim 1.5$~billion \emph{1-vertex}
triangulations from the closed census for additional study.

Generating exhaustive censuses of all possible inputs
requires sophisticated algorithms and significant computational resources.
The $n \leq 10$-tetrahedron census of all closed triangulations
first appeared in \cite{burton11-genus} (which also describes
some of the underlying algorithms).
The $n \leq 8$-tetrahedron census of all bounded triangulations
is new to this paper.  Moreover, we have constructed this bounded census
for $n=9$, with over 20~billion triangulations; however, we only use
$n \leq 8$ for our experiments because the subsequent analysis
of normal surfaces for $n=9$ remains out of our computational reach
(for $n=8$ this analysis already consumed years of CPU time).

\begin{table}[tb]
\centering
\small
\begin{tabular}{c|rr|rr|rr|rr|rr|rr}
Input &
\multicolumn{8}{c|}{Combinatorial complexity $\sigma(\tri)$} &
\multicolumn{4}{c}{Algebraic complexity $\kappa(\tri)$} \\
\cline{2-13}
size $n$ &
\multicolumn{2}{c|}{Closed} &
\multicolumn{2}{c|}{Closed 1-vertex} &
\multicolumn{2}{c|}{Bounded} &
\multicolumn{2}{c|}{Bounded, discs} &
\multicolumn{2}{c|}{Closed} &
\multicolumn{2}{c}{Bounded} \\
&
Max & Avg &
Max & Avg &
Max & Avg &
Max & Avg &
Max & Avg &
Max & Avg \\
\hline
1  &    3 &   2.0 &    2 &   1.7 &    7 &   5.0 &    7 &  4.0 &  1 & 1.0 &  1 & 1.0 \\
2  &    7 &   3.9 &    4 &   3.3 &   14 &   8.2 &   14 &  5.2 &  2 & 1.2 &  2 & 1.3 \\
3  &   11 &   5.5 &    8 &   4.9 &   35 &  14.0 &   27 &  7.0 &  3 & 1.5 &  3 & 1.7 \\
4  &   18 &   8.8 &   16 &   7.8 &   85 &  31.3 &   69 & 11.6 &  4 & 1.8 &  6 & 2.5 \\
5  &   36 &  13.3 &   32 &  12.0 &  236 &  69.5 &  176 & 20.2 &  7 & 2.1 & 12 & 3.4 \\
6  &   70 &  20.8 &   64 &  18.6 &  688 & 152.6 &  440 & 34.8 & 10 & 2.3 & 20 & 4.3 \\
7  &  144 &  32.2 &  128 &  28.8 & 1943 & 376.6 & 1109 & 61.7 & 16 & 2.6 & 36 & 5.8 \\
8  &  291 &  50.2 &  256 &  44.7 & 5725 & 947.4 & 2768 &112.4 & 26 & 2.9 & 65 & 7.5 \\
9  &  584 &  78.5 &  512 &  69.4 &      &       &      &      & 42 & 3.2 &    &     \\
10 & 1175 & 123.2 & 1024 & 108.2 &      &       &      &      & 68 & 3.6 &    &     \\ 
\hline
Growth &
$2.03^n$ & $1.56^n$ & $2^n$ & $1.56^n$ & $2.73^n$ & $2.23^n$ & $2.45^n$ &
$1.69^n$ & $1.62^n$ & $1.14^n$ & $1.81^n$ & $1.34^n$
\end{tabular}
\caption{Experimental worst-case and average-case results}
\label{tab-expt}
\end{table}

Table~\ref{tab-expt} summarises our experimental results, and gives
worst-case and average-case measurements (labelled \emph{Max} and \emph{Avg}
respectively) for
the quantities $\sigma(\tri)$ and $\kappa(\tri)$ in our various settings.
Each measurement is taken over the relevant census of triangulations from
Table~\ref{tab-census}.  The \emph{Closed} and \emph{Bounded} columns
refer to all closed or bounded triangulations respectively;
in the \emph{Closed 1-vertex} column we restrict our attention to
1-vertex triangulations of closed manifolds, and in the
\emph{Bounded, discs} column we only count vertex normal discs (not all
vertex normal surfaces).  We have also measured the algebraic complexity
$\kappa(\tri)$ in the 1-vertex and discs-only settings,
but we omit the details due to space constraints;
see the appendix for the details.

The final row of Table~\ref{tab-expt} gives a ``best estimate'' of the
exponential growth rate of each quantity with respect to $n$
(we just list the base of the exponential, ignoring any
coefficients or polynomial factors).
Most growth rates are estimated by linear regression%
    \footnote{Specifically, we take a weighted linear regression of
    $\log \sigma$ or $\log \kappa$ as a function of $n$.  The weights are
    taken to be $1,\ldots,n$, in order to limit the influence of
    anomalous small cases.},
though for cases where the worst cases matches a known family of
triangulations (see below) we give the corresponding known rate.

We can make some broad observations from Table~\ref{tab-expt}:
\begin{itemize}
    \item The average-case scenarios grow at a
    significantly slower rate than the worst-case scenarios,
    sometimes astonishingly so.
    This is consistent with past observations
    in which ``typical'' triangulations
    exhibit significantly smaller complexity properties than
    expected (see Section~\ref{sec:intro}).

    \item For closed manifolds, 1-vertex triangulations only give a very
    slight improvement: the worst case drops from the
    $\Theta(17^{n/4}) \simeq \Theta(2.03^n)$ family described in
    \cite{burton10-complexity} down to the $\Theta(2^n)$ family of
    Theorem~\ref{thm:pow2}.
    The closeness of these results is surprising, since
    the theoretical bounds on $\sigma(\tri)$ for 1-vertex triangulations
    are much smaller than the general bounds (see below), and
    algorithms for working with them are often much simpler
    \cite{jaco03-0-efficiency}.

    \item Bounded triangulations exhibit higher complexity
    properties than their closed counterparts.  For the combinatorial
    complexity $\sigma(\tri)$ this discrepancy is very pronounced---even
    the average case for bounded triangulations is well above
    the worst case for closed triangulations.
    This is again consistent with past experiences in working with
    normal surface algorithms \cite{burton12-ws}.
    Restricting our attention to normal discs (e.g., for unknot recognition)
    does alleviate this problem somewhat.
\end{itemize}

Table~\ref{tab-summary} compares the experimental behaviour of the
worst-case $\sigma(\tri)$ against its best known theoretical
lower and upper bounds (here ``lower bounds'' refers to families of
pathological triangulations with the highest known growth rate of
$\sigma(\tri)$, such as those constructed in
Section~\ref{sec:lower}).

Regarding $\sigma(\tri)$: the lower bound of
$\Theta(17^{n/4}) \simeq \Theta(2.03^n)$ is known from
\cite{burton10-complexity}, and the remaining three lower bounds are new
to this paper (Theorem~\ref{thm:pow2}, Corollary~\ref{cor:bdrygrowth} and
Theorem~\ref{thm:pathFamily}).  The first two upper bounds are taken from
\cite{burton11-asymptotic}; the final two $O(64^n)$ bounds are well known
but do not appear in the literature (for a proof we refer to the appendix).
Regarding $\kappa(\tri)$: all four lower bounds of
$\Theta([(1+\sqrt{5})/2]^n) \simeq \Theta(1.62^n)$
are new to this paper
(Theorems~\ref{thm:boundedLargeCoords} and \ref{thm:closedLargeCoords}),
and all four upper bounds are taken from \cite{burton10-tree}.

\begin{table}[tb]
\centering
\small
\begin{tabular}{l|rcr|rcr}
&
\multicolumn{3}{c|}{Combinatorial complexity $\sigma(\tri)$} &
\multicolumn{3}{c}{Algebraic complexity $\kappa(\tri)$} \\
\cline{2-7}
& \multicolumn{1}{c}{Lower}
& \multicolumn{1}{c}{Experimental}
& \multicolumn{1}{c|}{Upper}
& \multicolumn{1}{c}{Lower}
& \multicolumn{1}{c}{Experimental}
& \multicolumn{1}{c}{Upper} \\
& \multicolumn{1}{c}{bound}
& \multicolumn{1}{c}{growth}
& \multicolumn{1}{c|}{bound}
& \multicolumn{1}{c}{bound}
& \multicolumn{1}{c}{growth}
& \multicolumn{1}{c}{bound} \\
\hline
Closed &
$\Omega(2.03^n)$ &
\multicolumn{1}{c}{$\longleftarrow$} &
$O(14.556^n)$ &
$\Omega(1.62^n)$ &
\multicolumn{1}{c}{$\longleftarrow$} &
$O(3.17^n)$ \\
Closed 1-vertex &
$\Omega(2^n)$ &
\multicolumn{1}{c}{$\longleftarrow$} &
$O(4.852^n)$ &
$\Omega(1.62^n)$ &
\multicolumn{1}{c}{$\longleftarrow$} &
$O(3.17^n)$ \\
Bounded &
$\Omega(2.37^n)$ &
$\simeq 2.73^n$ &
$O(64^n)$ &
$\Omega(1.62^n)$ &
$\simeq 1.82^n$ &
$O(31.63^n)$ \\
Bounded, discs only &
$\Omega(2^n)$ &
$\simeq 2.45^n$ &
$O(64^n)$ &
$\Omega(1.62^n)$ &
\multicolumn{1}{c}{$\longleftarrow$} &
$O(31.63^n)$ \\
\end{tabular}
\caption{Summary of worst-case theoretical and experimental results}
\label{tab-summary}
\end{table}

Here we see that the experimental growth rates are much closer to
the lower bounds than the upper bounds; in particular, an arrow
($\longleftarrow$) indicates that the experimental worst-case growth rate is
\emph{identical} to the best lower bound (up to a constant factor).
This invites the following conjectures:

\begin{conjecture}
    For closed triangulations, the maximum number of vertex normal
    surfaces $\sigma(\tri)$ for any given $n$ grows at an asymptotic rate of
    $\Theta(17^{n/4}) \simeq \Theta(2.03^n)$, and for closed
    1-vertex triangulations this reduces to $\Theta(2^n)$.

    For closed triangulations as well as closed 1-vertex
    triangulations, the maximum coordinate of any vertex normal surface
    $\kappa(\tri)$ for any given $n$ grows at an asymptotic rate of
    $\Theta([(1+\sqrt{5})/2]^n) \simeq \Theta(1.62^n)$.
    For bounded triangulations, the maximum coordinate of any vertex normal
    disc for any given $n$ likewise grows at an asymptotic rate of
    $\Theta([(1+\sqrt{5})/2]^n) \simeq \Theta(1.62^n)$.
\end{conjecture}

For closed 1-vertex triangulations, our experimental data gives an
even stronger result:

\begin{theorem}
    \label{thm:binomialexpt}
    For closed 1-vertex triangulations, the maximum number of vertex
    normal surfaces $\sigma(\tri)$ for any given $n \leq 10$ is precisely
    $2^n$, and is attained by the binomial triangulations
    $\mathcal{A}_n$ as described in Section~\ref{sec:closedTrigs}.
\end{theorem}

\begin{conjecture}
    Theorem~\ref{thm:binomialexpt} is true for all positive integers $n$.
\end{conjecture}

\section{Discussion}
\label{sec:conc}

As noted in the introduction, although we consider close to 3~billion
distinct triangulations, they are all relatively small with less than or equal to $10$
tetrahedra.  Despite this, there are reasons to believe that our
experimental results might generalise.
Because we allow more flexible triangulations (not just simplicial
complexes), this census contains a rich diversity of 3-manifolds
(including 5114 distinct closed $\mathbb{P}^2$-irreducible manifolds
\cite{burton11-genus,matveev03-algms}).
Moreover, several of the patterns that we see in Table~\ref{tab-expt}
(in particular Theorem~\ref{thm:binomialexpt}) are established early on
and are remarkably consistent.

As an exception, in the case of arbitrarily bounded manifolds,
a greater number of tetrahedra seems to allow a better choice for the starting triangulation $\mathcal{G}_0$
in Lemma~\ref{lem:treeLemma} in order to obtain a higher exponential base.

As seen in Table~\ref{tab-summary}, there is still a long way to go
before lower bounds and upper bounds on worst-case complexities converge.
Not only does this paper produce the first
explicit lower bounds in several computationally important settings, but it
also gives strong experimental evidence that these lower bounds are close to
(or even exactly) tight.  This suggests that it is now the upper bounds
that require significant improvement, inviting new directions of
research with a rich interplay between topology, polytopes and
complexity theory.

%%%%%%%%%%%%%%%%%%%%%%%%%%%%%%%%%%%%%%%%%%%%%%%%%%%%%%%%%%%%%%%%%%%%%%%%
%
%   Bibliography
%
%%%%%%%%%%%%%%%%%%%%%%%%%%%%%%%%%%%%%%%%%%%%%%%%%%%%%%%%%%%%%%%%%%%%%%%%

%TODO Does Kuperberg have a to-appear reference? No (October 23rd)

\bibliographystyle{amsplain}
\bibliography{pure}

\newcommand{\noopsort}[1]{}
\providecommand{\bysame}{\leavevmode\hbox to3em{\hrulefill}\thinspace}
\providecommand{\MR}{\relax\ifhmode\unskip\space\fi MR }
% \MRhref is called by the amsart/book/proc definition of \MR.
\providecommand{\MRhref}[2]{%
  \href{http://www.ams.org/mathscinet-getitem?mr=#1}{#2}
}
\providecommand{\href}[2]{#2}
\begin{thebibliography}{10}

\bibitem{birman80-problems}
Joan~S. Birman, \emph{Problem list: Nonsufficiently large 3-manifolds}, Notices
  Amer. Math. Soc. \textbf{27} (1980), no.~4, 349.

\bibitem{burton09-convert}
Benjamin~A. Burton, \emph{Converting between quadrilateral and standard
  solution sets in normal surface theory}, Algebr. Geom. Topol. \textbf{9}
  (2009), no.~4, 2121--2174.

\bibitem{burton10-complexity}
\bysame, \emph{The complexity of the normal surface solution space}, SCG '10:
  Proceedings of the Twenty-Sixth Annual Symposium on Computational Geometry,
  ACM, 2010, pp.~201--209.

\bibitem{burton11-genus}
\bysame, \emph{Detecting genus in vertex links for the fast enumeration of
  3-manifold triangulations}, {ISSAC} 2011: Proceedings of the 36th
  International Symposium on Symbolic and Algebraic Computation, ACM, 2011,
  pp.~59--66.

\bibitem{burton11-asymptotic}
\bysame, \emph{Maximal admissible faces and asymptotic bounds for the normal
  surface solution space}, J. Combin. Theory Ser. A \textbf{118} (2011), no.~4,
  1410--1435.

\bibitem{burton12-regina}
\bysame, \emph{Computational topology with {R}egina: Algorithms, heuristics and
  implementations}, To appear in Geometry \& Topology Down Under, Amer. Math.
  Soc., \texttt{arXiv:\allowbreak 1208.2504}, 2012.

\bibitem{regina}
Benjamin~A. Burton, Ryan Budney, William Pettersson, et~al., \emph{Regina:
  Software for 3-manifold topology and normal surface theory},
  \texttt{http://\allowbreak regina.\allowbreak sourceforge.\allowbreak net/},
  1999--2012.

\bibitem{burton10-tree}
Benjamin~A. Burton and Melih Ozlen, \emph{A tree traversal algorithm for
  decision problems in knot theory and 3-manifold topology}, To appear in
  Algorithmica, \texttt{arXiv:\allowbreak 1010.6200}, 2010.

\bibitem{burton11-crosscap}
\bysame, \emph{Computing the crosscap number of a knot using integer
  programming and normal surfaces}, To appear in ACM Trans. Math. Software,
  \texttt{arXiv:\allowbreak 1107.2382}, 2011.

\bibitem{burton12-ws}
Benjamin~A. Burton, J.~Hyam Rubinstein, and Stephan Tillmann, \emph{The
  {W}eber-{S}eifert dodecahedral space is non-{H}aken}, Trans. Amer. Math. Soc.
  \textbf{364} (2012), no.~2, 911--932.

\bibitem{chiodo11-groups}
Maurice Chiodo, \emph{Finding non-trivial elements and splittings in groups},
  J. Algebra \textbf{331} (2011), 271--284.

\bibitem{dunfield11-spanning}
Nathan~M. Dunfield and Anil~N. Hirani, \emph{The least spanning area of a knot
  and the optimal bounding chain problem}, {SCG} '11: Proceedings of the
  Twenty-Seventh Annual Symposium on Computational Geometry, ACM, 2011,
  pp.~135--144.

\bibitem{dunfield06-random-covers}
Nathan~M. Dunfield and William~P. Thurston, \emph{Finite covers of random
  3-manifolds}, Invent. Math. \textbf{166} (2006), no.~3, 457--521.

\bibitem{haken61-knot}
Wolfgang Haken, \emph{Theorie der {N}ormalfl{\"a}chen}, Acta Math. \textbf{105}
  (1961), 245--375.

\bibitem{hass12-conp}
Joel Hass, \emph{New results on the complexity of recognizing the 3-sphere}, To
  appear in Oberwolfach Rep., 2012.

\bibitem{hass99-knotnp}
Joel Hass, Jeffrey~C. Lagarias, and Nicholas Pippenger, \emph{The computational
  complexity of knot and link problems}, J. Assoc. Comput. Mach. \textbf{46}
  (1999), no.~2, 185--211.

\bibitem{hass03-spanningdisks}
Joel Hass, Jack Snoeyink, and William~P. Thurston, \emph{The size of spanning
  disks for polygonal curves}, Discrete Comput. Geom. \textbf{29} (2003),
  no.~1, 1--17.

\bibitem{jaco03-0-efficiency}
William Jaco and J.~Hyam Rubinstein, \emph{0-efficient triangulations of
  3-manifolds}, J. Differential Geom. \textbf{65} (2003), no.~1, 61--168.

\bibitem{jaco95-algorithms-decomposition}
William Jaco and Jeffrey~L. Tollefson, \emph{Algorithms for the complete
  decomposition of a closed {$3$}-manifold}, Illinois J. Math. \textbf{39}
  (1995), no.~3, 358--406.

\bibitem{kirby78-calculus}
Robion Kirby, \emph{A calculus for framed links in {$S^{3}$}}, Invent. Math.
  \textbf{45} (1978), no.~1, 35--56.

\bibitem{kuperberg11-conp}
Greg Kuperberg, \emph{Knottedness is in {NP}, modulo {GRH}}, Preprint,
  \texttt{arXiv:\allowbreak 1112.0845}, November 2011.

\bibitem{lickorish97}
W.~B.~Raymond Lickorish, \emph{An introduction to knot theory}, Graduate Texts
  in Mathematics, no. 175, Springer, New York, 1997.

\bibitem{maher10-random-heegaard}
Joseph Maher, \emph{Random {H}eegaard splittings}, J. Topol. \textbf{3} (2010),
  no.~4, 997--1025.

\bibitem{markov60-insolubility}
A.~A. Markov, \emph{Insolubility of the problem of homeomorphy}, Proc.
  Internat. Congress Math. 1958, Cambridge Univ. Press, New York, 1960,
  pp.~300--306.

\bibitem{matveev03-algms}
Sergei Matveev, \emph{Algorithmic topology and classification of 3-manifolds},
  Algorithms and Computation in Mathematics, no.~9, Springer, Berlin, 2003.

\bibitem{orlik72-seifert}
Peter Orlik, \emph{Seifert manifolds}, Lecture Notes in Mathematics, no. 291,
  Springer-Verlag, Berlin, 1972.

\bibitem{reiner12-mfo-problems}
Vic Reiner and John Sullivan, \emph{Triangulations: Minutes of the open problem
  session}, To appear in Oberwolfach Rep., 2012.

\bibitem{rubinstein04-smallsfs}
J.~Hyam Rubinstein, \emph{An algorithm to recognise small {S}eifert fiber
  spaces}, Turkish J. Math. \textbf{28} (2004), no.~1, 75--87.

\end{thebibliography}

\newpage

\appendix
\section*{Appendix: Additional proofs}
\subsection*{Additional Preliminaries}

In order to fully understand the detailed proofs given in this appendix
we need some more terminology concerning triangulations and normal surface theory.

From Section \ref{sec:prelim} we know that normal surfaces within a triangulation
can be described by their normal coordinates which, due to the matching
equations and the non-negativity constraints,
lie within a high-dimensional polyhedral cone.
However, not all of the integer points in this cone
give rise to a normal surface. Those normal coordinates which in fact do belong to
a normal surface are called \emph{admissible}.\footnote{%
    Some authors define ``admissible'' to include all non-negative
    multiples of such points (so admissible vectors need not have
    integer coordinates) \cite{burton09-convert,burton11-asymptotic}.
    In this paper we restrict the notion of admissibility to
    integer vectors only.}
Given an admissible vector in this cone, the
corresponding normal surface can be reconstructed uniquely
(up to certain forms of isotopy).

The reason why not all normal coordinates are admissible is because
there is no way to insert two quadrilaterals of different types into 
the same tetrahedron without having a non-empty intersection 
(cf.~Figure~\ref{fig:normalSubsets}). Since normal surfaces are defined
to be embedded, it follows that for every tetrahedron there can be at most one
type of normal quadrilateral (though there may be many quadrilaterals of
this one type). Translated to normal coordinates, this
means that for every tetrahedron of a triangulation at most one of the
three coordinates accounting for the quadrilaterals may be non-zero.
This condition is called the {\em quadrilateral constraints}, and it can
be shown that the admissible points---that is, the normal coordinates
that do correspond to a normal surface---are precisely those integer points
in the solution cone that satisfy the quadrilateral constraints.

Inside the cone it is straightforward to define the
\emph{sum} of two normal surfaces
as the sum of their normal coordinates: if two normal coordinates satisfy
the matching equations and non-negativity constraints then their 
sum will also satisfy them. However, not all of these sums represent 
normal surfaces, since they need not satisfy the quadrilateral
constraints. Because of this we call two normal surfaces {\em compatible}
if the sum of their normal coordinates is admissible.
% Clearly, the quadrilateral constraints are not closed under summation,
% which is why not all normal surfaces are compatible.

\medskip
Due to the gluings of a triangulation, many vertices of the tetrahedra
may become identified together and, as a consequence, a small neighbourhood
of a vertex in a triangulation looks like a union of cusps of tetrahedra
which are themselves smaller tetrahedra.
Hence, the outer boundary (or frontier)
of this neighbourhood
is a collection of triangles which all are normal pieces.  Therefore
this outer boundary is a normal surface, which we call a
{\em vertex linking normal surface} or just a {\em vertex link}.
In a closed triangulation (which represents a closed manifold),
all vertex links must be spheres (otherwise the neighbourhood of some
vertex is not locally $\mathbb{R}^3$).
In a bounded triangulation (which represents a manifold with boundary), all vertex links must be spheres or discs.

\subsection*{The proof of Theorem \ref{thm:pow2}}

In this section, we give a full proof of a more detailed version of Theorem~\ref{thm:pow2}
from Section~\ref{sec:closedTrigs}:
%First, we note that each $\mathcal{A}_n$ is orientable, and so all normal
%surfaces in $\mathcal{A}_n$ must be orientable also.
%We now prove a more detailed version of Theorem~\ref{thm:pow2}:

\begin{theorem}
For each $n \ge 1$, there are $\binom{n}{k}$ vertex normal surfaces
in $\mathcal{A}_n$ of genus $k$. Therefore there are $\sigma(\mathcal A_n)=2^{n}$ vertex normal surfaces.
\end{theorem}

% TODO: Say what the four vertex normal surfaces are (e.g., these come
% from the triangles, as opposed to quadrilaterals, within the single tetrahedron).

\begin{proof}
Consider a single tetrahedron $\Delta_i$ with two faces
identified as in Section \ref{sec:closedTrigs}. Define $\alpha$, $\beta$
and $\gamma$ to be loops surrounding vertices $0$, $1$ and $2$
respectively. The projective solution space has four vertex normal
surfaces $\bf{a}$, $\bf{b}$, $\bf{c}$ and $\bf{d}$.
All of these surfaces have a non-empty boundary which consists of a combination of
$\alpha$, $\beta$ and $\gamma$ and all of them are compatible.
Therefore we do not have to consider the quadrilateral constraints, which
simplifies the rest of the argument.

Let
$\mathbf{a}_i$, $\mathbf{b}_i$, $\mathbf{c}_i$, $\mathbf{d}_i$
denote the surfaces and
$\alpha_i$, $\beta_i$, $\gamma_i$ the boundary curves in the corresponding
$\Delta_i$. Any normal surface in $\mathcal A_n$ can be written as
$$\sum_{i=1}^{n} \tau_{a_i} \mathbf{a}_{i} + \tau_{b_i} \mathbf{b}_{i} +
\tau_{c_i} \mathbf{c}_{i} + \tau_{d_i} \mathbf{d}_{i}$$ 
for constants
$\tau_{a_i},\tau_{b_i},\tau_{c_i},\tau_{d_i}\ge0$ for $i=1,...,n$. Following
the argument in \cite{burton10-complexity}, there are only two ways
in which a vertex normal surface can meet each $\Delta_i$ and satisfy
the matching equations, namely
$\mathbf{u}_i=\mathbf{b}_i+\mathbf{c}_i+\mathbf{d}_i$ and
$\mathbf{v}_i=\mathbf{a}_i+\mathbf{d}_i$, both with boundary
$\alpha_i + \beta_i + \gamma_i$. It follows that every normal surface in
$\mathcal A_n$ can be described by 
$$ \sum_{i=1}^{n}p_{i} \mathbf{u}_{i} + q_{i}\mathbf{v}_{i} $$ 
where each $p_{i},q_{i} \ge 0$ and where
$\epsilon := p_{1}+q_{1}=p_{2}+q_{2}=...=p_{n}+q_{n}$ is an integer. 

It follows from the above that each normal surface meets the boundary of $\Delta_i$,
$i = 1, \ldots , n$, in $\epsilon (\alpha_i + \beta_i + \gamma_i)$.
Each normal surface with $\epsilon > 1$ is the sum of two other normal surfaces
in $\mathcal A_n$ and thus is not a vertex normal surface. On the other hand,
if $\epsilon = 1$ it can be shown that the only rational linear combinations 
$p_{i} + q_{i} = 1$, $i = 1, \ldots , n$, resulting in a normal surface of 
$\mathcal A_n$ must be $p_{i} = 1$ and $q_{i} = 0$ or $p_{i} = 0$ and $q_{i} = 1$.

Therefore, the projective
solution space has $2^{n}$ admissible vertices, corresponding to the
$2^{n}$ combinations of either $p_{i} = 1$ and $q_{i} = 0$ or $p_{i} = 0$ and $q_{i} = 1$ 
at each $\Delta_i$. The binomial coefficients easily follow from a simple Euler
characteristic argument since each surface $\mathbf{u}_i$ is a disc
and each $\mathbf{v}_i$ is
a punctured torus.
\end{proof}

\subsection*{Proof of Lemma \ref{lem:treeLemma}}

Here we give a full proof of Lemma~\ref{lem:treeLemma} from Section~\ref{sec:bdTrigs}, 
which we restate below. 

\setcounter{theorem}{1} % XTODO: Make sure this stays correct.
\begin{lemma}
	Suppose $\mathcal{G}_0$ is a bounded triangulation with $n_0$ tetrahedra,
    $f_0$ is a boundary face of $\mathcal{G}_0$ such that not all
	vertices of $f_0$ are identified in $\mathcal{G}_0$, and
    $c_0$ is one of the three normal arc types on $f_0$.
    If there are $\alpha_0$ vertex normal surfaces in $\mathcal{G}_0$
    that meet $f_0$ in at least one arc of type $c_0$ but in no other normal arc
    types, then $\mathcal{G}_0$ can be extended to a family of
    triangulations $\{\mathcal{G}_k\}$ in which the number of
    vertex normal surfaces grows at a rate of $\Omega(\beta^n)$,
    where $\beta = \alpha_0^{1/(n_0+1)}$.
\end{lemma}

\begin{proof}
The proof consists of constructing such a family of triangulations. To do so, 
we take two copies of $\mathcal{G}_0$ and 
an additional tetrahedron $\Delta_1$ and join both copies along 
their faces $f_0$ to $\Delta_1$ (see Figure \ref{fig:attachingMap} where $k=1$),
yielding a new triangulation $\mathcal{G}_1$.
Any pair of the $\alpha_0$ vertex normal surfaces meeting
$f_0$ in a copy of normal arcs of type $c_0$ 
can be combined	such that this linear combination has a unique extension throughout $\Delta_1$ meeting 
a third face $f_1$ of $\Delta_1$ in only one of the three normal arcs
$c_1$. An argument exploiting the 
uniqueness of this extension shows that any such pairing yields a new vertex normal surface and there are 
$\alpha_1 = \alpha_0^2$ such pairings. The assumption that not all
vertices of the boundary face $f_0$ are identified
in $\mathcal{G}_0$ is necessary to show that the resulting triangulation
still represents a bounded manifold (i.e., all vertex links are discs or
spheres).

In a next step, two copies of $\mathcal{G}_1$ are combined together in an
analogous way to create $\mathcal{G}_2$, and so on. After
$k$ steps we have a full binary tree of depth $k$ where the ``leaves'' of the tree are all
copies of $\mathcal{G}_0$. In each step the number of tetrahedra is
doubled plus one extra ``root tetrahedron'' is added. On the
other hand, we at least square the number of vertex normal surfaces. These observations 
put together are sufficient to yield the desired result.
	\begin{figure}[tb]
		\begin{center}
		   \includegraphics[width=.5\textwidth]{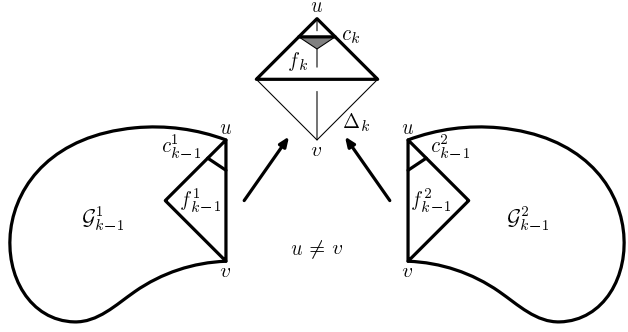}
		   \caption{Attaching two copies of $\mathcal{G}_{k-1}$ to the root tetrahedron $\Delta_k$. \label{fig:attachingMap}}
		\end{center}
	\end{figure}

	\medskip
    More precisely: let $\mathcal{G}_k$ be the triangulation after the
    $k$-th step with $n_k$ tetrahedra, with a designated boundary
	face $f_k$ with a fixed normal arc type $c_k$ of the root tetrahedron $\Delta_k$ which will be glued to the new root 
	tetrahedron $\Delta_{k+1}$, and with at least $\alpha_k$ vertex normal surfaces meeting $f_k$ only in copies of $c_k$
	(see Figure \ref{fig:attachingMap}).

	Then we have $n_{k+1} = 2 n_k +1$ and $\alpha_{k+1} \geq (\alpha_k)^2$. Now let $\beta$ be a real number
	such that $\alpha_0 \geq \beta^{n_0+1}$ and let us assume that $\alpha_k \geq \beta^{n_k+1}$. Since 
	$\alpha_{k+1} \geq (\alpha_k)^2$ it then follows that 
	\begin{eqnarray}
		\alpha_{k+1} &\geq& (\beta^{n_k+1})^2 \nonumber \\
		&=& \beta^{2n_k+2} \nonumber \\
		&=& \beta^{n_{k+1}+1} \nonumber
	\end{eqnarray}
	and thus $\mathcal{G}_{k+1}$ has at least $\beta^{n_{k+1}+1}$ vertex normal surfaces. To finish the induction argument we now simply
	choose $\beta = \alpha_0^{1/(n_0+1)}$.

	It remains to show that (i) $\mathcal{G}_k$ is a triangulation of a bounded $3$-manifold and (ii) any of the $\alpha_k$ 
	normal surfaces is in fact a vertex normal surface.

	To show that $\mathcal{G}_k$ represents a bounded $3$-manifold, we first have to take a closer look at how the two copies of
	$\mathcal{G}_{k-1}$ (denoted by $\mathcal{G}_{k-1}^1$ and $\mathcal{G}_{k-1}^2$) have to be attached to $\Delta_k$: the two boundary triangles of type 
	$f_{k-1}$ (denoted by $f_{k-1}^1$ and $f_{k-1}^2$) are joined to $\Delta_k$ along an edge $e$.
	The gluing has to be in a way that the endpoints of $e$ are not
    previously identified in \emph{both} $f_{k-1}^1$ and $f_{k-1}^2$,
	and that the normal arcs of
	type $c_{k-1}$ of both triangles (denoted by $c_{k-1}^1$ and $c_{k-1}^2$) are next to the same vertex of $\Delta_k$ which 
	thus has to be an endpoint of $e$ (see Figure \ref{fig:attachingMap}).
    Note that this is
	always possible if not all vertices of $f_{k-1}^1$ are identified
    together and likewise with $f_{k-1}^2$.

	Since $\mathcal{G}_0 = \mathcal{G}$ is a bounded triangulation, in order to prove that $\mathcal{G}_k$ is a bounded triangulation it suffices to assume that 
	$\mathcal{G}_{k-1}$ is a bounded manifold and then show that all vertex links of vertices of the root tetrahedron $\Delta_k$ are still 
	triangulations of the disc after attaching $\mathcal{G}_{k-1}^1$ and $\mathcal{G}_{k-1}^2$ to $\Delta_k$. 

	If none of the vertices of $f_{k-1}^i$ are identified in $\mathcal{G}_{k-1}^i$, $i=1,2$, none of the vertices of $\Delta_k$ 
	will be identified. 
	Thus, all vertex links are either just copies of the vertex links in $\mathcal{G}_{k-1}^1$ and $\mathcal{G}_{k-1}^2$ with an additional triangle attached or two 
	vertex linking discs attached to each other along an additional triangle (see Figure \ref{fig:DiscLinks} on the left). 

	If two of the vertices of $f_{k-1}^1$ and $f_{k-1}^2$ are identified,
    and if these are the vertices opposite the arcs
    $c_{k-1}^1$ and $c_{k-2}^2$, then
    one of these vertices in each face must be disjoint from $e$ in $\Delta_k$.
	In this case, three vertices of $\Delta_k$ are identified to the same vertex and the vertex link consists  
	of the two disjoint vertex links in $\mathcal{G}_{k-1}^1$ and
    $\mathcal{G}_{k-1}^2$ joined together along a triangle (the normal
    piece near the endpoint of $e$ opposite of $c_{k-1}^1$ and $c_{k-2}^2$), and two additional 
	disjoint triangles which are added at opposite ends of the new disc. The result is still a disc (see Figure \ref{fig:DiscLinks} on the right).
	The fourth vertex of $\Delta_k$ simply connects two discs as described in the previous situation.
	It follows that $\mathcal{G}_k$ must represent a bounded $3$-manifold,
    and the new face $f_k$ likewise does not have all three vertices
    identified.
	\begin{figure}[tb]
		\begin{center}
		   \includegraphics[width=.85\textwidth]{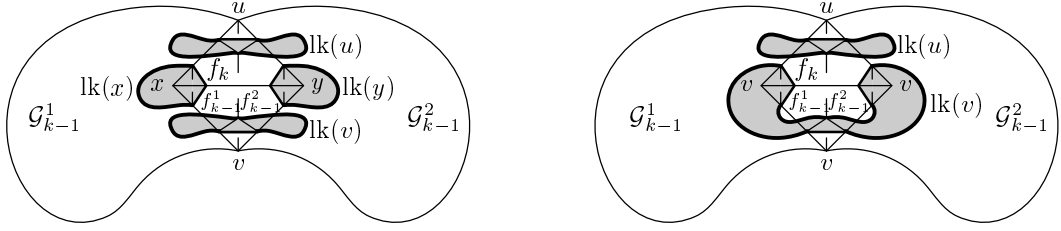}
		   \caption{Left: no two vertices of $f_{k-1}^1$ or $f_{k-1}^2$ are identified. Right: two vertices of $f_{k-1}^1$ and $f_{k-1}^2$ 
			are identified. $\operatorname{lk}(v)$ denotes the link of vertex $v$. \label{fig:DiscLinks}}
		\end{center}
	\end{figure}

	If two of the vertices of $f_{k-1}^1$ and $f_{k-1}^2$ are identified,
    and if these include the vertices inside arcs $c_{k-1}^1$ and $c_{k-2}^2$,
    then we must be careful how we attach the two copies of
    $\mathcal{G}_{k-1}$ to $\Delta_k$: in one copy $\mathcal{G}_{k-1}^1$
    we ensure that the two identified vertices join to the common edge $e$,
    and in the other copy $\mathcal{G}_{k-1}^2$ we only allow one of
    these vertices to join to $e$.  By a similar argument, this ensures
    that the new vertex links in $\mathcal{G}_k$ are discs, and that the
    new boundary face $f_k$ again does not have all three vertices
    identified.

	\medskip
	To prove that all $\alpha_{k}$ normal surfaces constructed in the
    $k$-th step are in fact vertex normal surfaces, consider
	two normal surfaces $\mathbf{u}_1$ and $\mathbf{u}_2$ meeting $f_{k-1}$ in exactly $x_1$ and $x_2$ normal arcs of type $c_{k-1}$. 
	Let $m$ be the smallest common multiple of $x_1$ and $x_2$, then
	$\frac{m}{x_1} \, \mathbf{u}_1 + \frac{m}{x_2} \, \mathbf{u}_2$ can be extended to a normal surface in $\mathcal{G}_k$ by inserting $m$ matching triangles into $\Delta_k$
	(cf. Figure \ref{fig:attachingMap}). 
	We will denote this normal surface by $\mathbf{u}$. Now assume that $\mathbf{u}$ can be written as $\mathbf{u} = \lambda \mathbf{s} + \mu \mathbf{t}$ for 
	two normal surfaces $\mathbf{s}$ and $\mathbf{t}$ of $\mathcal{G}_k$ and 
	two rational numbers $\lambda$ and $\mu$. By construction, the restriction of this linear combination to $\mathcal{G}_{k-1}^1$ or $\mathcal{G}_{k-1}^2$ 
	is a multiple of a vertex normal surface. Hence, inside $\mathcal{G}_{k-1}^1$ or $\mathcal{G}_{k-1}^2$ both $\mathbf{s}$ and 
	$\mathbf{t}$ are rational multiples of $\mathbf{u}$. Moreover, since $\mathbf{u}$ inside $\Delta_k$
	only consists of $m$ triangles of a given type, $\mathbf{s}$ and $\mathbf{t}$ have to be multiples of $\mathbf{u}$ inside $\Delta_k$ as well and by the matching equations
	for $\mathbf{s}$ and $\mathbf{t}$ at $f_{k-1}^1$ and $f_{k-1}^2$ it follows that $\mathbf{s}$ and $\mathbf{t}$ are multiples of $\mathbf{u}$ 
	throughout the entire triangulation $\mathcal{G}_k$. Thus, 
	$\mathbf{u}$ is a vertex normal surface of $\mathcal{G}_k$.
\end{proof}

\subsection*{Construction of triangulation $\mathcal{G}$}

In Section~\ref{sec:bdTrigs} we refer to an $11$-tetrahedra triangulation
$\mathcal{G}$ as a starting point for the family $\{ \mathcal{G}_k \}$ of 
bounded triangulations in Corollary \ref{cor:bdrygrowth}. Here we present this 11-tetrahedron
triangulation in full.

\medskip
The triangulation $\mathcal{G}$, consisting of $n=11$ tetrahedra $\Delta_i = i(0123)$, 
$i = 0 , \ldots , 10$, is given by the gluings in Table~\ref{tab:g},
where missing entries denote boundary faces. Face $f_0$ is given by $f_0 = 0(012)$ and the normal 
arc type $c_0$ is the one around vertex $0$ of $f_0$, hence $c_0 = 0$.
Using \texttt{Regina} \cite{burton12-regina,regina} for vertex 
normal surface enumeration yields $\sigma (\mathcal{G}) = 61\,526$ and $\alpha_0 = 31\,643$. 
The face pairing graph of $\mathcal{G}$ is shown in Figure~\ref{fig:face_pairing_graphs}.

\begin{table}[tb]
\[\begin{array}{c|rrrr}
	\Delta_i & i(012) & i(013) & i(023) & i(123) \\
	\hline
	0&&1(012)&2(021)& \\
	1&0(013)&3(012)&4(021)& \\
	2&0(032)&5(012)&6(021)& \\
	3&1(013)&7(231)&7(023)& \\
	4&1(032)&8(231)&8(023)& \\
	5&2(013)&9(231)&9(023)& \\
	6&2(032)&10(231)&10(023)& \\
	7&&&3(023)&3(301) \\
	8&&&4(023)&4(301) \\
	9&&&5(023)&5(301) \\
	10&&&6(023)&6(301) 
\end{array}\]
\caption{The triangulation $\mathcal{G}$}
\label{tab:g}
\end{table}

\subsection*{Proof of Theorem \ref{thm:pathFamily}}

Here we give a full proof of Theorem~\ref{thm:pathFamily} from Section~\ref{sec:bdTrigs}, 
which we restate below. 

\setcounter{theorem}{3} % XTODO: Make sure this stays correct.
\begin{theorem}
For each $n \ge 1$, $\mathcal P_{n}$ has $2^{n+1}+\frac{(n+1)(n+2)}{2}
\in \Theta(2^n)$ vertex normal discs.
\end{theorem}

\begin{proof}
First of all, let $\mathcal P_{i}$ be the triangulation after the $i$-th
step with $i$ tetrahedra with a designated boundary face $f_i$ of
$\Delta_i$, and with
fixed normal arcs $a_i$, $b_i$ and $c_i$ on this face which will be
glued to tetrahedron $\Delta_{i+1}$, as shown in
Figure~\ref{fig:pathFamily}.
Also let $\alpha_i$, $\beta_i$ and $\gamma_i$ be the number of vertex normal surfaces meeting face $f_i$ in $a_i$, $b_i$ and $c_i$ respectively. Finally, let $\delta_i$ be the number of vertex normal surfaces not meeting face $f_i$.

	\begin{figure}[tb]
		\begin{center}
		   \includegraphics[width=.45\textwidth]{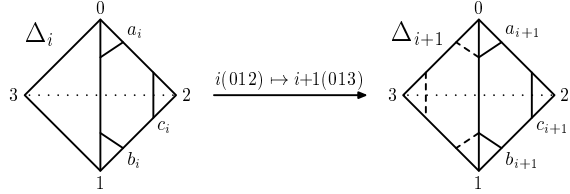}
		   \caption{Notation for the normal arcs of the path family. \label{fig:pathFamily}}
		\end{center}
	\end{figure}

First, let us consider the case of a single tetrahedron $\Delta_{1}$. There are seven admissible vertices in the projective solution space of $\Delta_{1}$ corresponding to seven vertex normal discs: four single triangles around each vertex and three single quadrilaterals. There are two surfaces meeting at each type of normal arc in face $f_1$ and one surface without any normal arcs in this face. Therefore $\alpha_1=2$, $\beta_1=2$, $\gamma_1=2$ and $\delta_1=1$.

Now, let us construct explicit formulae for each of the sequences
$\alpha_i$, $\beta_i$, $\gamma_i$ and $\delta_i$. In tetrahedron
$\Delta_{i+1} = (i+1)(0123)$, normal surfaces with normal arcs $a_{i+1}$ either have a triangle around vertex $0$ or a quadrilateral separating edge $03$ from edge $12$, where the notation is as indicated in Figure \ref{fig:pathFamily}. The triangle meets face $f_{i}$ in normal arc $a_{i}$ and the quadrilateral meets $f_{i}$ in normal arc $b_{i}$. Therefore the number of normal surfaces $\alpha_{i+1}$ is exactly $\alpha_{i} + \beta_{i}$. A similar argument shows that $\beta_{i+1} = \beta_{i} + \alpha_{i}$, and since $\alpha_1=\beta_1=2$, it follows that $\alpha_i=\beta_i=2^i$.

Normal surfaces with normal arcs $c_{i+1}$ on the other hand either have a triangle around vertex $2$ or a quadrilateral separating edge $01$ from edge $23$ (see Figure \ref{fig:pathFamily}). The triangle does not meet $f_{i}$ and the quadrilateral meets $f_{i}$ in normal arc $c_{i}$. Therefore the number of normal surfaces $\gamma_{i+1} = \gamma_{i} + 1$. Therefore since $\gamma_1=2$ we have $\gamma_i=i+1$.

Now consider $\delta_{i+1}$, the number of normal surfaces not meeting face $f_{i+1}$. These surfaces either have a triangle around vertex $3$ or do not meet face $f_{i}$ (see Figure \ref{fig:pathFamily}). Since there are exactly $\delta_i$ normal surfaces disjoint to $f_{i}$ and the triangle around vertex $3$ meets face $f_{i}$ in a normal arc of type $c_{i}$, we have $\delta_{i+1} = \gamma_{i}+\delta_{i}$. It follows that $\delta_i=\frac{i(i+1)}{2}$, since $\delta_1=1$ and $\gamma_i=i+1$.

We note at this point that all vertex normal surfaces in $\mathcal{P}_n$
have been accounted for,
since any normal surface that meets the face $f_{i+1}$ in
more than one arc can be expressed as a sum of the surfaces described above.
Therefore the total number of vertex normal surfaces in $\mathcal P_{n}$ is 

$$\alpha_n+\beta_n+\gamma_n+\delta_n=2^n+2^n+(n+1)+\frac{n(n+1)}{2}=2^{n+1}+\frac{(n+1)(n+2)}{2} .$$ 

It remains to show that each of the above vertex normal surfaces is a disc. At the $(i+1)$-st step, there is only one triangle or one quadrilateral in $\Delta_{i+1}$ glued to one normal arc in $f_{i}$. Therefore, we either add one vertex, two edges and one face in the case of the triangle or two vertices, three edges and one face in the case of the quadrilateral to the surfaces in $\mathcal P_{i}$. In both cases the Euler characteristic does not change in the $(i+1)$-st step. Since in tetrahedron $\Delta_1$ all seven surfaces are discs, it follows that every surface in $\mathcal P_{n}$ is a disc.
\end{proof}

\subsection*{Proof of Theorem \ref{thm:boundedLargeCoords}}

In the following, we give a full proof of Theorem~\ref{thm:boundedLargeCoords} 
from Section~\ref{sec:largeCoords}, which we restate below. 

\begin{theorem}
	There is a family $\mathcal{B}_{n}$ of bounded $1$-vertex
    triangulations with $n$ tetrahedra,
    where $\mathcal{B}_n$ contains a vertex normal disc $\mathbf{d}_{n}$ with
    maximum coordinate $\operatorname{F}_{n+1}$, 
	where $\operatorname{F}_k$ denotes the $k$-th Fibonacci number.
\end{theorem}

\begin{proof}
	The family of triangulations $\mathcal{B}_{n}$ consists of layered solid tori of type
	$\operatorname{LST}
    (\operatorname{F}_{n+1},\operatorname{F}_{n+2},\operatorname{F}_{n+3})$
    and each vertex 
	normal surface $\mathbf{d}_{n}$ will be the
    corresponding meridian disc with algebraic complexity of at least $\operatorname{F}_{n+1}$.

	\medskip
	In order to prove this we will start with the base case $\mathcal{B}_{1} = \operatorname{LST}(1,2,3)$. Figure \ref{fig:lst123}
	shows a one tetrahedron triangulation of $\operatorname{LST}(1,2,3)$ with the meridian disc having maximum coordinate $\operatorname{F}_{2} = 1$. 
	Furthermore, by using \texttt{Regina} 
	\cite{burton12-regina,regina} we can check that the meridian disc $\mathcal{B}_{n}$ is a vertex normal surface of $\mathcal{B}_{n}$.

	Now, let us assume that $\mathcal{B}_{n}$ is a $n$-tetrahedra triangulation of 
	$\operatorname{LST} (\operatorname{F}_{n+1},\operatorname{F}_{n+2},\operatorname{F}_{n+3})$
	containing the meridian disc $\mathbf{d}_{n}$ as a vertex normal surface with maximum
	coordinate $\operatorname{F}_{n+1}$, and intersecting the boundary edges in $\operatorname{F}_{n+1}$, 
	$\operatorname{F}_{n+2}$ and $\operatorname{F}_{n+3}$ points (see Figure \ref{fig:largeLST} on the left). 
	\begin{figure}[tb]
    		\begin{center}
        	    \includegraphics[width=.65\textwidth]{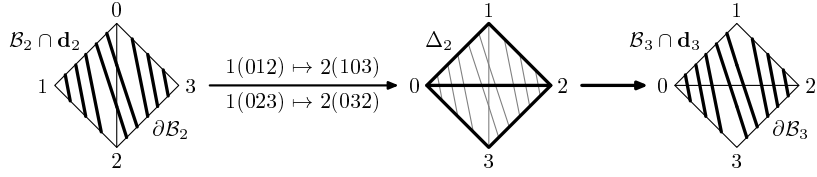}
	 	   \caption{Building $\mathcal{B}_{n}$ in the case $n=2$. Left: boundary of $\mathbf{d}_{2}$ in $\partial \mathcal{B}_{2}$. 
				Center: the $3$rd tetrahedron is glued to $\mathcal{B}_{2}$ Right: boundary of $\mathbf{d}_{3}$ in 
				$\partial \mathcal{B}_{3}$. \label{fig:largeLST}}
     		\end{center}
	\end{figure}

	To construct $\mathcal{B}_{n+1}$ and $\mathbf{d}_{n+1}$ we now glue the $(n+1)$-st tetrahedron to 
	$\operatorname{LST} (\operatorname{F}_{n+1},\operatorname{F}_{n+2},\operatorname{F}_{n+3})$ such that the
	boundary edge with $\operatorname{F}_{n+1}$ intersecting points becomes an internal edge (see Figure \ref{fig:largeLST} in
 	the center). Note that $\mathcal{B}_{n+1}$ still triangulates a solid torus. 
	There are $\operatorname{F}_{n+2}$ parallel normal arcs both in the upper left and in the lower right corner
	(in each case $\operatorname{F}_{n+1}$ arcs coming from triangles and $\operatorname{F}_n$ arcs 
	coming from quadrilaterals of $\mathbf{d}_{n}$).
	For each of them we will insert a normal triangle into the $(n+1)$-st tetrahedron of $\mathcal{B}_{n+1}$, intersecting 
	the boundary of $\mathcal{B}_{n+1}$ in two arcs and the
	newly inserted boundary edge $e$ in one point. For each of the remaining $\operatorname{F}_{n+1}$ pairs of arcs in the middle 
	of the boundary of $\mathcal{B}_{n}$ we insert $\operatorname{F}_{n+1}$ parallel normal quadrilaterals into the $(n+1)$-st tetrahedron of $\mathcal{B}_{n+1}$,
	each intersecting the boundary of $\mathcal{B}_{n+1}$ in two normal arcs and edge $e$ in one point.

	Altogether, we extend the meridian disc of $\mathcal{B}_{n}$ to a new normal surface by adding $2 \operatorname{F}_{n+2}$ triangles and 
	$\operatorname{F}_{n+1}$ quadrilaterals, hence $\operatorname{F}_{n+4}$
    normal pieces each intersecting $e$ in one point. The new normal surface  
	has two types of normal triangles with each $\operatorname{F}_{n+2}$ copies, and thus algebraic complexity of at least $\operatorname{F}_{n+2}$. 
	Starting from $\mathbf{d}_{n}$, each of these extensions applied one
	after another turns the previous disc into a new disc. As a consequence, the normal
	surface constructed by attaching all $\operatorname{F}_{n+4}$ normal
    pieces to $\mathbf{d}_{n}$ is still a disc and will be denoted by
	$\mathbf{d}_{n+1}$. The boundary curve of $\mathbf{d}_{n+1}$ is still simply closed
	and non-contractible in the boundary of $\mathcal{B}_{n+1}$, hence, it is the meridian disc of $\mathcal{B}_{n+1}$ intersecting the
	boundary edges in $\operatorname{F}_{n+2}$, $\operatorname{F}_{n+3}$ 
	and $\operatorname{F}_{n+4}$ points.
    By construction, $\mathcal{B}_{n+1}$ is a layered solid torus of type
	$\operatorname{LST} (\operatorname{F}_{n+2},\operatorname{F}_{n+3},\operatorname{F}_{n+4})$ with $n+1$ tetrahedra 
	(see Figure \ref{fig:largeLST} on the right).

	It remains to show that $\mathbf{d}_{n+1}$ is a vertex normal surface. This is equivalent to the statement that
	for any rational combination 
	\begin{equation}
		\label{eq:rationalComb}
		\mathbf{d}_{n+1} = \lambda \mathbf{s} + \mu \mathbf{t}
	\end{equation}
	with $\mathbf{s}$ and $\mathbf{t}$ normal surfaces in $\mathcal{B}_{n+1}$ and $\lambda, \mu \in \mathbb{Q}$, the normal surfaces $\mathbf{s}$ 
	and $\mathbf{t}$ must be rational multiples of $\mathbf{d}_{n+1}$.

	Equation \ref{eq:rationalComb} is valid for every subset of normal coordinates in $\mathcal{B}_{n+1}$. In particular,
	it holds if we restrict $\mathbf{s}$, $\mathbf{t}$ and $\mathbf{d}_{n+1}$ to the normal coordinates associated to the first $n$ tetrahedra of
	$\mathcal{B}_{n+1}$. However, this restriction applied to $\mathbf{d}_{n+1}$ yields $\mathbf{d}_{n}$ which is a vertex normal surface by 
	assumption. It follows that $\mathbf{s} = \theta \mathbf{d}_{n}$ and $\mathbf{t} = \psi \mathbf{d}_{n}$ in $\mathcal{B}_{n}$ for some rational numbers
	$\theta$ and $\psi$.

	Extending $\mathcal{B}_{n}$ to $\mathcal{B}_{n+1}$, we glue the $(n+1)$-st tetrahedron onto the two boundary triangles 
	of $\mathcal{B}_{n}$, thus adding seven new variables and six new matching constraints to the projective solution space of $\mathcal{B}_{n}$
	in order to enumerate the vertex normal surfaces of $\mathcal{B}_{n+1}$. 
	Following the labelling of Figure \ref{fig:largeLST} on the right, we will denote the seven normal coordinates
	of the $(n+1)$-st tetrahedron by
	$$ (t_0, t_1, t_2, t_3 \,|\, q_{01}, q_{02}, q_{03} ) $$
	where $t_X$ denotes the triangle type isolating vertex $X$ from the rest of the tetrahedron and $q_{XY}$ 
	denotes the quadrilateral type isolating edge $XY$. Taking the normal coordinates from the boundary curve of
	$\mathbf{d}_{n}$ into account, this results in the following equations
	\begin{eqnarray}
		0 &=& t_2 + q_{03} \nonumber \\
		0 &=& t_3 + q_{03} \nonumber \\
		\operatorname{F}_{n+2} &=& t_0 + q_{01} \nonumber \\
		\operatorname{F}_{n+2} &=& t_1 + q_{01} \nonumber \\
		\operatorname{F}_{n+1} &=& t_2 + q_{02} \nonumber \\
		\operatorname{F}_{n+1} &=& t_3 + q_{02} . \nonumber 
	\end{eqnarray}

	It follows immediately that $t_2 = t_3 = q_{03} = 0$, and $q_{02} = \operatorname{F}_{n+1}$, $q_{01}=0$ by the quadrilateral
	constraints and thus $t_0 = t_1 = \operatorname{F}_{n+2}$. Hence, the matching equations assure that $\mathbf{d}_{n}$ has a unique
	extension and the same holds for the restricted versions of $\mathbf{s}$ and $\mathbf{t}$ for some rational multiple of the matching
	equations. It follows that $\mathbf{s}$ and $\mathbf{t}$ are rational multiples of $\mathbf{d}_{n+1}$ and $\mathbf{d}_{n+1}$ is a vertex normal surface.  
\end{proof}

\subsection*{The triangulation $\mathcal{E}$}

In Section~\ref{sec:largeCoords} we refer to a $4$-tetrahedron triangulation
$\mathcal{E}$ acting as a kind of ``plug'' for the bounded family $\mathcal{B}_n$
in order to obtain the closed family $\mathcal{C}_n$ with large algebraic
complexity. Here we present this 4-tetrahedron
triangulation including its vertex normal surfaces $\mathbf{s}$ and
$\mathbf{t}$ in detail.

\medskip
The triangulation $\mathcal{E}$ is given by the
$4$ tetrahedra $\Delta_i = i (0123)$, $i = 0 , \ldots , 3$, and the
following table of gluings:
$$
\begin{array}{c|cccc}
	\Delta_i & i(012) & i(013) & i(023) & i(123) \\
	\hline
	0&2(231)&1(230)&2(023)&1(123) \\
	1&3(012)&2(102)&0(301)&0(123) \\
	2&1(103)&3(230)&0(023)&0(201) \\
	3&1(012)&\partial_1&2(301)&\partial_2 
\end{array}
$$
where $\partial_1 = 3(013)$ and $\partial_2 = 3(123)$ denote the two boundary faces of $\mathcal{E}$ (see Figure \ref{fig:sAndtApp}).
The face pairing graph of $\mathcal{E}$ is shown in Figure \ref{fig:face_pairing_graphs}.

Enumeration of all vertex normal surfaces of $\mathcal{E}$ using \texttt{Regina} \cite{burton12-regina,regina} 
yields $13$ surfaces, $12$ of them bounded and one of them closed. However, in the following we 
will only take a closer look at two of the bounded surfaces which will be denoted by $\mathbf{s}$ 
and $\mathbf{t}$. Surface $\mathbf{s}$ is a cylinder with two boundary components
(see Figure \ref{fig:sAndtApp} on the left), $\mathbf{t}$ is a M\"obius strip with one boundary 
component (see Figure \ref{fig:sAndtApp} on the right).
In particular, both surfaces have Euler characteristic $0$.
\begin{figure}[tb]
	\begin{center}
       	    \includegraphics[width=.3\textwidth]{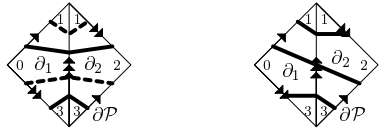}
 	   \caption{Left: boundary pattern of $\mathbf{s}$, two disjoint circles. Right: boundary pattern of $\mathbf{t}$ \label{fig:sAndtApp}}
	\end{center}
\end{figure}

The surfaces are given by the following standard normal coordinates (cf. proof of Theorem \ref{thm:boundedLargeCoords} for
more about the notation of standard normal coordinates):
$$
\begin{array}{lcllll}
	\mathbf{s} &=&(0,1,1,0\,|\,0,0,0),&(0,1,0,0\,|\,0,1,0),&(0,1,0,0\,|\,0,0,1),&(0,2,0,2\,|\,0,0,0) \\
	\mathbf{t} &=&(0,0,0,0\,|\,1,0,0),&(1,1,0,0\,|\,0,0,0),&(1,1,0,0\,|\,0,0,0),&(0,1,0,1\,|\,0,0,1),\\
\end{array}
$$
and boundary patterns
$$
\begin{array}{lcll}
	\partial \mathbf{s} &=&(0,2,2),&(2,0,2) \\
	\partial \mathbf{t} &=&(0,2,1),&(1,0,2),\\
\end{array}
$$
where the first triple counts the normal arcs of boundary face $\partial_1$ and the second triple the ones of boundary face 
$\partial_2$. The triples are ordered starting with the normal arcs around the smallest vertex label.

The surfaces $\mathbf{s}$ and $\mathbf{t}$ are compatible, meaning that their normal coordinates can be added to each other yielding a new normal
surface. Since the Euler characteristic is additive under this summing
operation it follows that all linear combinations of $\mathbf{s}$ and $\mathbf{t}$ will have Euler characteristic $0$. Moreover, any linear combination
of $\mathbf{s}$ and $\mathbf{t}$ is connected if and only if it is orientable with two boundary components or 
non-orientable with one boundary component. There cannot be any interior closed connected components 
since the only closed vertex normal surface is not compatible with $\mathbf{t}$ and a scalar multiple of $\mathbf{s}$ cannot contain
closed connected components. Hence, any connected linear combination of $\mathbf{s}$ and $\mathbf{t}$ is a cylinder or a M\"obius strip where the
latter is true if and only if the linear combination has an odd number of copies of $\mathbf{t}$.

\medskip
Keeping this in mind, simple addition of the boundary patterns of $\mathbf{s}$ and $\mathbf{t}$ shows that 
$$\operatorname{F}_{n-4} \, \mathbf{t} + \frac{1}{2}\operatorname{F}_{n-5} \, \mathbf{s},$$ 
$n \geq 5$, yields the boundary pattern of the meridian disc $\mathbf{d}_{n-4}$ from the layered solid torus family $\mathcal{B}_{n-4}$ from Theorem \ref{thm:boundedLargeCoords} above. Moreover, since the boundary patterns of $\mathbf{s}$ and $\mathbf{t}$ are not multiples of each other, this is the only linear combination
of $\mathbf{s}$ and $\mathbf{t}$ with this property. Now, since $\mathbf{s}$ contains odd coordinates, this boundary pattern bounds a surface if and only if $\operatorname{F}_{n-5}$ 
is even, or equivalently $n \equiv 2 \bmod 3$. In this case $\operatorname{F}_{n-4}$ is odd and since the boundary pattern is connected it bounds a M\"obius strip
which will be denoted by $\operatorname{M}_{n}$. If, on the other hand $n \equiv 0,1 \bmod 3$, the boundary pattern  
$$2\operatorname{F}_{n-4} \, \mathbf{t} + \operatorname{F}_{n-5} \,\mathbf{s}$$
has two connected components and since $2 \operatorname{F}_{n-4}$ is always even, it bounds a cylinder $\operatorname{Cyl}_{n}$.

\subsection*{Proof of Theorem \ref{thm:closedLargeCoords}}

In the following, we give a full proof of Theorem~\ref{thm:boundedLargeCoords} 
from Section~\ref{sec:largeCoords}, which we restate below. 

\begin{theorem}
	There is a family $\mathcal{C}_{n}$ of closed $1$-vertex triangulations with $n$ tetrahedra, $n \geq 5$, each 
	containing a vertex normal surface with maximum coordinate at least $\operatorname{F}_{n-3}$
	if $n \equiv 2 \bmod 3$ or at least $2\operatorname{F}_{n-3}$ otherwise.
\end{theorem}

\begin{proof}
	We will construct the family $\mathcal{C}_{n}$ by gluing $\mathcal{B}_{n-4}$ and $\mathcal{E}$ along their boundary 
	components. Then, if $n \equiv 2 \bmod 3$, the meridian disc 
	$\mathbf{d}_{n-4}$ glued with $\operatorname{M}_{n}$ yields a vertex normal projective 
	plane or, if $n \equiv 0,1 \bmod 3$, twice $\mathbf{d}_{n-4}$ glued
	with $\operatorname{Cyl}_{n}$ yields a vertex normal 
	sphere of $\mathcal{C}_{n}$ and the maximum normal coordinates are as stated. We will call these surfaces $\mathcal{S}_{n}$.

	\medskip
	It remains to show that $\mathcal{S}_{n}$ is a vertex normal surface of $\mathcal{C}_{n}$ for all $n \geq 5$. To see this, recall that
	$\mathbf{d}_{n-4}$ is a vertex normal surface of $\mathcal{B}_{n-4}$ for all $n \geq 5$. Hence, it suffices to show that 
	$\mathbf{d}_{n}$ ($2 \mathbf{d}_{n}$) has a unique extension in $\mathcal{C}_{n}$ yielding $\mathcal{S}_{n}$ if 
	$n \equiv 2 \bmod 3$ ($n \equiv 0,1 \bmod 3$) and
	that $\mathcal{S}_{n}$ has the smallest integer normal coordinates amongst all rational multiples of $\mathcal{S}_{n}$.

	From the section above we know that there is at least one linear
    combination of vertex normal surfaces in $\mathcal{E}$ realising 
	a valid extension of $\mathbf{d}_{n}$ or $2\mathbf{d}_{n}$ yielding $\mathcal{S}_{n}$. Using the classification 
	of all vertex normal 
	surfaces of $\mathcal{E}$ we can see that due to conflicting quadrilateral constraints and matching equations there are only
	four vertex normal surfaces in $\mathcal{E}$ which may occur in such a linear combination. Due to further compatibility
	constraints such a linear combination can be shown to be a linear combination of $\mathbf{s}$ and 
	$\mathbf{t}$ and the only closed vertex
	normal surfaces of $\mathcal{E}$. Following the observations made in the section above it follows that there is no other
	linear combination of vertex normal surfaces and hence the extension of 
	$\mathbf{d}_{n}$ ($2\mathbf{d}_{n}$) in $\mathcal{E}$ is unique.

	Altogether it follows that $\mathcal{S}_{n}$ is a vertex normal surface.
\end{proof}

\subsection*{An upper bound for $\sigma (\tri)$}

The following statement (mentioned in Section~\ref{sec:expt})
is well known but does not appear in the literature, and so we give the
simple proof here.

\setcounter{theorem}{10} % XTODO: Make sure this stays correct.
\begin{lemma}
	\label{lem:upperBound}
	Let $\tri$ be a triangulation with $n$ tetrahedra. Then
	$\sigma (\tri) \leq 64^n$.
\end{lemma}

\begin{proof}
	It is known that each vertex normal surface is uniquely defined by
    its zero set,
	i.\ e., the set of normal coordinates which are zero (see \cite{burton11-asymptotic}).
	Let $(t_0,t_1,t_2,t_3 \, | \, q_{01}, q_{02}, q_{03})$ be the set of normal coordinates of a tetrahedron 
	of $\tri$. Then each of the $t_i$ can be either zero or non-zero which leaves us with $2^4 = 16$ 
	choices, and at most one of the $q_{0j}$ can be non-zero due to the
    quadrilateral constraints which leaves
	us with $4$ additional choices. All together we have $4 \cdot 2^4 = 64$ distinct zero sets 
	per tetrahedron which yields the result.
\end{proof}

\end{document}